\documentclass[10pt,a4paper,reqno]{amsart}

\usepackage[utf8]{inputenc}
\usepackage[english]{babel}
\usepackage{amsmath,amssymb,amsthm,amsfonts}
\usepackage{mathtools}
\usepackage{geometry}
\usepackage{microtype}
\usepackage{listings}
\usepackage{xcolor}
\usepackage{hyperref}
\usepackage{graphicx}
\usepackage{array}
\usepackage{longtable}
\usepackage{booktabs}
\lstdefinestyle{compactcert}{basicstyle=\ttfamily\scriptsize,columns=fullflexible,breaklines=true,keepspaces=true,showstringspaces=false,frame=single,framerule=0.2pt,xleftmargin=0.5em,aboveskip=0.6em,belowskip=0.6em}
\lstdefinestyle{compactscript}{basicstyle=\ttfamily\scriptsize,columns=fullflexible,breaklines=true,keepspaces=true,showstringspaces=false,frame=single,framerule=0.2pt,xleftmargin=0.5em,aboveskip=0.6em,belowskip=0.6em,language=Python}
\hypersetup{colorlinks=true,linkcolor=blue,citecolor=blue,urlcolor=blue}
\newtheorem{theorem}{Theorem}
\newtheorem{lemma}[theorem]{Lemma}
\newtheorem{proposition}[theorem]{Proposition}
\newtheorem{corollary}[theorem]{Corollary}
\newtheorem{definition}[theorem]{Definition}
\newtheorem{remark}[theorem]{Remark}

\newcommand{\cA}{\mathcal A}
\newcommand{\cB}{\mathcal B}
\newcommand{\cC}{\mathcal C}
\newcommand{\cD}{\mathcal D}
\newcommand{\cE}{\mathcal E}
\newcommand{\cF}{\mathcal F}
\newcommand{\cG}{\mathcal G}
\newcommand{\cHh}{\mathcal H}
\newcommand{\cI}{\mathcal I}
\newcommand{\cJ}{\mathcal J}
\newcommand{\1}{\mathbf{1}}
\newcommand{\ZZ}{\mathbb Z}
\newcommand{\Loop}{\operatorname{loop}}
\newcommand{\Leb}{\mathcal L}

\title{On the independence number of de Bruijn graphs}
\date{}

\author{Pietro Majer}
\address{Dipartimento di Matematica, Universit\`a di Pisa, Largo Bruno Pontecorvo 5, 56127 Pisa, Italy}
\email{pietro.majer@unipi.it}

\author{Matteo Novaga}
\address{Dipartimento di Matematica, Universit\`a di Pisa, Largo Bruno Pontecorvo 5, 56127 Pisa, Italy}
\email{matteo.novaga@unipi.it}


\begin{document}

\begin{abstract}
We derive the asymptotic formula $\alpha(k,q)=\lambda_{k-1}q^k+o(q^k)$,
where $\alpha(k,q)$ is the independence number of the de Bruijn graph $B(k,q)$,
and $\lambda_{k-1}$ is a constant arising from a variational problem on the unit $(k-1)$-dimensional cube. 
When $k=4$, we show the bounds
$91/240\le \lambda_3\le 11/28$.
For odd prime $k$, we analyse the binary case $q=2$ via a phase reduction on rotation orbits. For $k=11,13,17$ this yields compact orbit-marker certificates for optimal constructions. Combined with a lifting theorem by Lichiardopol, these certificates give exact formulas for $\alpha(11,q)$, $\alpha(13,q)$, and $\alpha(17,q)$ for all $q\ge2$, extending the known cases $k=3,5,7$.
\end{abstract}

\maketitle
\tableofcontents

\section{Introduction}

De Bruijn graphs are a classical object in combinatorics, graph theory, and information theory. Their origin goes back to de Bruijn's paper \cite{deB46}, while the Eulerian perspective was developed further by van Aardenne-Ehrenfest and de Bruijn in \cite{AvdB51}. Standard references include Fredricksen's survey \cite{Fred82}, the necklace construction of Fredricksen and Maiorana \cite{FM78}, Ralston's expository article \cite{Ral82}, the survey chapter of Du, Cao, and Hsu on de Bruijn and Kautz digraphs \cite{DCH96}, and the recent monograph of Etzion \cite{Etz24}.

In this paper we study the independence number of de Bruijn graphs. This problem lies at the intersection of graph theory and coding theory: independent sets in de Bruijn graphs encode collections of words that avoid prescribed overlap patterns, and in several cases they are closely related to comma-free codes. Exact formulas are known only in a limited number of lengths, while the general asymptotic behaviour for fixed word length is governed by the variational constants introduced by Trotter and Winkler \cite{TW98} and further analysed by the authors in \cite{MN12}.

For integers $k,q\ge 2$, we denote by $B(k,q)$ the de Bruijn digraph with vertex set $[q]^k$, where
\[
[q]:=\{0,1,\dots,q-1\},
\]
and with directed edges
\[
x_1x_2\dots x_k \longrightarrow x_2x_3\dots x_k y,
\qquad y\in[q].
\]
Thus, for us the first parameter is the word length and the second is the alphabet size. We notice that other conventions can be found in the literature: for instance, the underlying simple graph is denoted $UB(q,k)$ in \cite{Lich06}, while the same graph appears as $B(q,k)$ in \cite{CCT10}.

We denote by $\alpha(k,q)$ the independence number of the \emph{simple} graph obtained from $B(k,q)$ after deleting self-loops. Two vertices
\[
u=x_1x_2\dots x_k,
\qquad
v=y_1y_2\dots y_k
\]
are adjacent in the underlying graph if and only if one of the overlap relations
\[
x_2x_3\dots x_k=y_1y_2\dots y_{k-1}
\qquad\text{or}\qquad
y_2y_3\dots y_k=x_1x_2\dots x_{k-1}
\]
holds. 
The graph $B(k,q)$ has exactly $q$ looped vertices, namely the constant words
\[
0^k,\ 1^k,\ \dots,\ (q-1)^k.
\]
When these self-loops are retained we write $\alpha_{\Loop}(k,q)$ for the corresponding independence number. Since only the $q$ constant vertices are affected, one has
\[
\alpha_{\Loop}(k,q)\le \alpha(k,q)\le \alpha_{\Loop}(k,q)+q.
\]
In particular, for fixed $k$ the looped and loopless models have the same leading asymptotic coefficient as $q\to\infty$.

The independence problem for de Bruijn graphs has both a continuum asymptotic side and a finite, explicitly certifiable side, and the two viewpoints reinforce each other.

First, we show that for every fixed $k\ge 2$ the leading term of $\alpha(k,q)$, as $q\to\infty$, is given exactly by a variational constant $\lambda_{k-1}$. More precisely, as $q\to +\infty$ one has
\[
\alpha(k,q)=\lambda_{k-1}q^k+o(q^k).
\]
This identifies the asymptotic independence problem on de Bruijn graphs with the variational problem studied in \cite{TW98,MN12}. 

Second, we focus on the first unresolved case, namely $k=4$. Here the aim is not an exact formula, but a sharper understanding of the extremal coefficient. We prove explicit bounds on $\lambda_3$, combining a soft upper bound from a seven-cycle inequality with a constructive lower bound obtained from an explicit base configuration and an inductive dyadic lift.

Third, we return to exact finite constructions in the binary prime cases $k=11$, $k=13$, and $k=17$. The phase reduction on rotation orbits shows that an extremal example is determined by one phase choice on each non-trivial orbit; the remaining task is finite and computer-verifiable. We record compact orbit-marker certificates for all three cases and deduce exact formulas for the independence number for every alphabet size $q\ge 2$.

The case $k=4$ sits between several lengths for which the independence number is already understood exactly.

\begin{theorem}[see \cite{Lich06,CCT10}]\label{thm:length3}
There holds
\[
\begin{aligned}
&\alpha(2,2)=2, &&\alpha_{\Loop}(2,2)=1,&
\\
&\alpha(2,3)=3, &&\alpha_{\Loop}(2,3)=2,
\\
&\alpha(2,q)=\alpha_{\Loop}(2,q)=\Bigl\lfloor\frac{q^2}{4}\Bigr\rfloor &&\text{for every $q\ge4$,}
\\
&\alpha(3,q)=\frac{q^3-q}{3}+1, &&\alpha_{\Loop}(3,q)=\frac{q^3-q}{3}.
\\
&\alpha(5,q)=\frac{2(q^5-q)}{5}+1, &&\alpha_{\Loop}(5,q)=\frac{2(q^5-q)}{5},
\\
&\alpha(7,q)=\frac{3(q^7-q)}{7}+1, &&\alpha_{\Loop}(7,q)=\frac{3(q^7-q)}{7}.
\end{aligned}
\]
Moreover, if $k\ge3$ is an odd prime, then for every $q\ge2$ we have
\[
\alpha(k,q)\le \frac{(k-1)(q^k-q)}{2k}+1,
\qquad
\alpha_{\Loop}(k,q)\le \frac{(k-1)(q^k-q)}{2k}.
\]
\end{theorem}

In Theorems~\ref{thm:k11-binary}, \ref{thm:k13-binary}, and \ref{thm:k17-binary} below, we extend the previous result to the next prime cases $k=11$, $k=13$, and $k=17$, and prove that
\[
\alpha(11,q)=\frac{5(q^{11}-q)}{11}+1,
\quad
\alpha_{\Loop}(11,q)=\frac{5(q^{11}-q)}{11},
\]
and
\[
\alpha(13,q)=\frac{6(q^{13}-q)}{13}+1,
\quad
\alpha_{\Loop}(13,q)=\frac{6(q^{13}-q)}{13},
\]
while the same method gives
\[
\alpha(17,q)=\frac{8(q^{17}-q)}{17}+1,
\quad
\alpha_{\Loop}(17,q)=\frac{8(q^{17}-q)}{17}.
\]
By contrast, no exact formula is presently known for $\alpha(4,q)$. This makes the length $4$ case the first genuinely open instance, and one of the main motivations of the present work is to determine explicit bounds for its leading asymptotic coefficient.

\smallskip

For every integer $m\ge 1$ and every measurable set $A\subseteq[0,1]^m$, we now define
\begin{equation}\label{eqvar}
\Phi_m(A):=\int_{[0,1]^{m+1}} \1_A(x_1,\dots,x_m)\bigl(1-\1_A(x_2,\dots,x_{m+1})\bigr)\,dx_1\cdots dx_{m+1},
\end{equation}
and set
\begin{equation}\label{eqlam}
\lambda_m:=\sup\{\Phi_m(A): A\subseteq[0,1]^m\text{ measurable}\}.
\end{equation}
These constants were first introduced in \cite{TW98}, with a different definition, and the variational characterization in \eqref{eqvar} was later given in \cite{MN12}. As stated above,
these constants also govern the independence number of de Bruijn graphs for large $q$: for every fixed $k\ge2$, the asymptotic coefficient of $\alpha(k,q)$ is exactly $\lambda_{k-1}$.

Trotter and Winkler showed in \cite{TW98} the strict monotonicity of the sequence $\lambda_m$, and proved the bounds
\[
\frac m{2m+2}\le \lambda_m\le \frac m{2m+1},
\]
together with the identities $\lambda_1=1/4$, $\lambda_2=1/3$ and the inequality $\lambda_5\ge 27/64$.
The authors of this paper showed in \cite{MN12} that
\[
\lambda_m=\frac m{2m+2}
\qquad\text{for every even }m,
\]
and therefore for every odd word length $k$ one obtains the exact asymptotic coefficient
\[
\frac{\alpha(k,q)}{q^k}\longrightarrow \lambda_{k-1}=\frac{k-1}{2k}.
\]
Moreover, Proposition~\ref{prop:odd-no-qkm1} below yields a refined upper bound for odd $k$, while preserving the explicit lower bound coming from the elementary construction:
\[
\frac{k-1}{2k}q^k-O(q^{k-2})
\le \alpha(k,q)
\le \frac{k-1}{2k}q^k+O(q^{\delta(k)}),
\]
where the parameter $\delta(k)$ is defined in Proposition~\ref{prop:odd-no-qkm1}: it equals $1$ if $k$ is prime and $k/p$ if $k$ is composite, with $p$ the smallest prime divisor of $k$. In particular, for odd prime $k$ the error term in the upper bound is of order $O(q)$, consistently with the exact formulas quoted above for $k=3,5,7,11,13,17$.

The even-length case is more delicate. Already for $k=4$, the general bounds imply only
\[
\frac38\le \lambda_3<\frac25.
\]
Our second main result improves this interval to
\[
\frac{91}{240}\le \lambda_3\le \frac{11}{28},
\]
and consequently
\[
\frac{91}{240}q^4+o(q^4)\le \alpha(4,q)\le \frac{11}{28}q^4+o(q^4).
\]
Thus, although the exact value of $\lambda_3$ remains open, the admissible range becomes substantially narrower.

\smallskip

For the reader's convenience, we summarise the main contributions of the paper.

\begin{enumerate}
\item[1)] We prove that for every fixed $k\ge2$ the independence number of the de Bruijn graph satisfies
\[
\alpha(k,q)=\lambda_{k-1}q^k+o(q^k),
\]
thereby identifying the exact asymptotic coefficient with the continuum variational constant $\lambda_{k-1}$.

\item[2)] In the case $k=4$, we establish the explicit bounds
\[
\frac{91}{240}\le \lambda_3\le \frac{11}{28}.
\]
The upper bound is obtained from a combinatorial inequality on a $7$-cycle, while the lower bound comes from an inductive dyadic construction starting from an explicit finite configuration.

\item[3)] For binary words of prime length, we introduce a phase reduction on rotation orbits. Combined with certified constructions, this yields exact formulas for $\alpha(11,q)$, $\alpha(13,q)$, and $\alpha(17,q)$ for all $q\ge2$.
\end{enumerate}

\smallskip

The paper is organised as follows. In Section~\ref{sec:asymptotic} we reformulate the discrete extremal problem in a way that makes the connection with the variational constants $\lambda_m$ transparent, and we prove the asymptotic formula for every fixed $k$. Section~\ref{sec:k4} is devoted to the case $k=4$: we first derive the upper bound for $\lambda_3$, then construct the dyadic lower bound, and finally translate these bounds into corresponding estimates for $\alpha(4,q)$. To keep the main argument readable, the large explicit $q=16$ data defining the base configuration are recorded separately in Appendix~\ref{app:q16}, where we record both the explicit data defining the base configuration and the finite verification attached to it. In Section~\ref{sec:odd-prime-binary} we recall the phase reduction for odd prime binary lengths and use it to record the certified cases $k=11$, $k=13$, and $k=17$. Appendix~\ref{app:certificates} records the compact certificates for $k=11$, $k=13$, and $k=17$, and gives a short verifier script.

\subsection*{Acknowledgements}
The authors acknowledge support from the MIUR Excellence Department Project awarded to the Department of Mathematics, University of Pisa, CUP I57G22000700001. M.N. is a member of INDAM-GNAMPA.

\section{The asymptotic coefficient}\label{sec:asymptotic}

Fix integers $k\ge2$ and $q\ge1$. Independent sets in $B(k,q)$ are naturally encoded by subsets of $[q]^{k-1}$.

\begin{definition}
For $S\subseteq[q]^{k-1}$, define
\[
\mathcal I_{k,q}(S):=
\bigl\{(x_1,\dots,x_k)\in[q]^k:
(x_1,\dots,x_{k-1})\in S,\ (x_2,\dots,x_k)\notin S\bigr\}.
\]
Equivalently, $\mathcal I_{k,q}(S)$ consists of the words of length $k$ whose prefix of length $k-1$ belongs to $S$ and whose suffix of length $k-1$ does not. We also set
\[
N_{k,q}(S):=|\mathcal I_{k,q}(S)|.
\]
\end{definition}

\begin{proposition}\label{prop:exact-reduction}
Define
\[
M_{k,q}:=\max\bigl\{N_{k,q}(S):S\subseteq[q]^{k-1}\bigr\}.
\]
Then
\[
M_{k,q}=\alpha_{\Loop}(k,q).
\]
Moreover,
\[
M_{k,q}\le \alpha(k,q)\le M_{k,q}+q,
\]
so in particular
\[
\alpha(k,q)=M_{k,q}+O(q)
\qquad (k\ \text{fixed},\ q\to\infty).
\]
\end{proposition}

\begin{proof}
Fix $S\subseteq[q]^{k-1}$. $\mathcal I_{k,q}(S)$ is an independent set in the graph with loops retained. Let
\[
x=(x_1,\dots,x_k)\in \mathcal I_{k,q}(S).
\]
By definition,
\[
(x_1,\dots,x_{k-1})\in S,
\qquad
(x_2,\dots,x_k)\notin S.
\]
Suppose that a successor
\[
y=(x_2,\dots,x_k,t)
\qquad (t\in[q])
\]
also belonged to $\mathcal I_{k,q}(S)$. Then, by the defining condition for $\mathcal I_{k,q}(S)$, its prefix $(x_2,\dots,x_k)$ would have to lie in $S$, a contradiction. Thus $x$ is adjacent to no successor in $\mathcal I_{k,q}(S)$. The same argument applied to predecessors shows that no predecessor of $x$ can belong to $\mathcal I_{k,q}(S)$ either. Hence no two vertices of $\mathcal I_{k,q}(S)$ are adjacent.

A looped constant vertex cannot occur in $\mathcal I_{k,q}(S)$, because a word $a^k$ would require simultaneously $a^{k-1}\in S$ and $a^{k-1}\notin S$. Therefore $\mathcal I_{k,q}(S)$ is independent in the looped model, and so
\[
\alpha_{\Loop}(k,q)\ge |\mathcal I_{k,q}(S)|=N_{k,q}(S).
\]
Taking the maximum over $S$ yields
\[
\alpha_{\Loop}(k,q)\ge M_{k,q}.
\]

Conversely, let $J$ be an independent set in the graph with loops retained. Define
\[
S_J:=\bigl\{(x_1,\dots,x_{k-1})\in[q]^{k-1}:\exists x_k\in[q]\text{ such that }(x_1,\dots,x_k)\in J\bigr\}.
\]
If $(x_1,\dots,x_k)\in J$, then $(x_1,\dots,x_{k-1})\in S_J$ by construction. We have $(x_2,\dots,x_k)\notin S_J$, for otherwise there would exist $t\in[q]$ such that $(x_2,\dots,x_k,t)\in J$, and these two vertices would be adjacent, contradicting the independence of $J$. Therefore every vertex of $J$ belongs to $\mathcal I_{k,q}(S_J)$, hence
\[
|J|\le |\mathcal I_{k,q}(S_J)|=N_{k,q}(S_J)\le M_{k,q}.
\]
Taking the maximum over all looped independent sets $J$ gives
\[
\alpha_{\Loop}(k,q)\le M_{k,q}.
\]
Combining the two inequalities proves the identity $M_{k,q}=\alpha_{\Loop}(k,q)$.

Finally, passing from the looped graph to the simple graph only affects the $q$ constant vertices $a^k$. Hence one may gain at most $q$ additional vertices when the loops are deleted, and therefore
\[
M_{k,q}=\alpha_{\Loop}(k,q)\le \alpha(k,q)\le \alpha_{\Loop}(k,q)+q=M_{k,q}+q.
\]
\end{proof}
Thus $M_{k,q}$ is the basic discrete quantity. Once its asymptotics are known, those of $\alpha(k,q)$ follow from the additive error term $q$.

For $\mathbf u=(u_1,\dots,u_{k-2})\in[q]^{k-2}$, we define
\[
I_S(\mathbf u):=\#\bigl\{a\in[q]:(a,u_1,\dots,u_{k-2})\in S\bigr\},
\]
\[
O_S(\mathbf u):=\#\bigl\{b\in[q]:(u_1,\dots,u_{k-2},b)\in S\bigr\}.
\]
When $k=2$, the index set $[q]^0$ is understood as the singleton $\{\emptyset\}$, so $I_S(\emptyset)=O_S(\emptyset)=|S|$.

A direct counting argument gives
\begin{equation}\label{eq:Nsum-general}
N_{k,q}(S)=\sum_{\mathbf u\in[q]^{k-2}} I_S(\mathbf u)\bigl(q-O_S(\mathbf u)\bigr).
\end{equation}
Indeed, for each fixed middle block $\mathbf u$, a word counted by $N_{k,q}(S)$ is obtained by choosing a prefix letter $a$ such that $(a,\mathbf u)\in S$ and a suffix letter $b$ such that $(\mathbf u,b)\notin S$.


For $S\subseteq[q]^{k-1}$ and $(a_1,\dots,a_{k-1})\in[q]^{k-1}$ we let
\[
Q_{a_1\dots a_{k-1}}^{(q)}:=\prod_{j=1}^{k-1}\Bigl[\frac{a_j}{q},\frac{a_j+1}{q}\Bigr)\subseteq[0,1]^{k-1},
\]
and 
\[
A_S:=\bigcup_{(a_1,\dots,a_{k-1})\in S} Q_{a_1\dots a_{k-1}}^{(q)}\subseteq[0,1]^{k-1}.
\]
Similarly, for $(a_1,\dots,a_k)\in[q]^k$ let
\[
Q_{a_1\dots a_k}^{(q)}:=\prod_{j=1}^{k}\Bigl[\frac{a_j}{q},\frac{a_j+1}{q}\Bigr)\subseteq[0,1]^k.
\]

\begin{lemma}\label{lem:correspondence}
For every $S\subseteq[q]^{k-1}$ one has
\[
\Phi_{k-1}(A_S)=\Lambda_{k,q}(S):=\frac{N_{k,q}(S)}{q^k}.
\]
\end{lemma}

\begin{proof}
Partition $[0,1]^k$ into the $q^k$ cubes $Q_{a_1\dots a_k}^{(q)}$. On such a cube the integrand in \eqref{eqvar}, with $m=k-1$, is constant and equals
\[
\1_S(a_1,\dots,a_{k-1})\bigl(1-\1_S(a_2,\dots,a_k)\bigr).
\]
Since each cube has volume $q^{-k}$, summing over all cubes gives
\[
\Phi_{k-1}(A_S)=q^{-k}\sum_{(a_1,\dots,a_k)\in[q]^k}
\1_S(a_1,\dots,a_{k-1})\bigl(1-\1_S(a_2,\dots,a_k)\bigr)
=\frac{N_{k,q}(S)}{q^k}.
\]
\end{proof}

\begin{lemma}\label{lem:l1-continuity}
For all measurable sets $A,B\subseteq[0,1]^{k-1}$,
\[
|\Phi_{k-1}(A)-\Phi_{k-1}(B)|
\le 2\,\|\1_A-\1_B\|_{L^1([0,1]^{k-1})}
=2|A\triangle B|.
\]
\end{lemma}

\begin{proof}
Write
\begin{align*}
\Phi_{k-1}(A)-\Phi_{k-1}(B)
&=
\int_{[0,1]^k}
\Bigl(\1_A(x_1,\dots,x_{k-1})(1-\1_A(x_2,\dots,x_k)) \\
&\hspace{4em}-\1_B(x_1,\dots,x_{k-1})(1-\1_B(x_2,\dots,x_k))\Bigr)
\,dx_1\cdots dx_k \\
&=
\int_{[0,1]^k}
\bigl(\1_A(x_1,\dots,x_{k-1})-\1_B(x_1,\dots,x_{k-1})\bigr)
\bigl(1-\1_A(x_2,\dots,x_k)\bigr)
\,dx \\
&\qquad+
\int_{[0,1]^k}
\1_B(x_1,\dots,x_{k-1})
\bigl(\1_B(x_2,\dots,x_k)-\1_A(x_2,\dots,x_k)\bigr)
\,dx.
\end{align*}
Taking absolute values and using $0\le \1_A,\1_B\le 1$ gives
\begin{align*}
|\Phi_{k-1}(A)-\Phi_{k-1}(B)|
&\le \int_{[0,1]^k}\bigl|\1_A(x_1,\dots,x_{k-1})-\1_B(x_1,\dots,x_{k-1})\bigr|\,dx \\
&\qquad+
\int_{[0,1]^k}\bigl|\1_A(x_2,\dots,x_k)-\1_B(x_2,\dots,x_k)\bigr|\,dx.
\end{align*}
In each integral one variable is free over an interval of length $1$, so both terms equal $\|\1_A-\1_B\|_{L^1([0,1]^{k-1})}$. This completes the proof.
\end{proof}

\begin{theorem}\label{thm:limit-exists}
For every fixed $k\ge2$,
\[
\lim_{q\to\infty}\frac{M_{k,q}}{q^k}=\lambda_{k-1}.
\]
Consequently,
\[
\alpha(k,q)=\lambda_{k-1}q^k+o(q^k).
\]
The same asymptotic formula also holds for $\alpha_{\Loop}(k,q)$.
\end{theorem}

\begin{proof}
Fix $q$ and let $S\subseteq[q]^{k-1}$. By Lemma~\ref{lem:correspondence},
\[
\Phi_{k-1}(A_S)=\frac{N_{k,q}(S)}{q^k}.
\]
Since $\lambda_{k-1}$ is the supremum of $\Phi_{k-1}$ over all measurable sets, this gives
\[
\frac{M_{k,q}}{q^k}\le \lambda_{k-1}
\qquad\text{for every }q,
\]
and hence
\[
\limsup_{q\to\infty}\frac{M_{k,q}}{q^k}\le \lambda_{k-1}.
\]
Conversely, let $\varepsilon>0$. Choose a measurable set $A\subseteq[0,1]^{k-1}$ such that
\[
\Phi_{k-1}(A)>\lambda_{k-1}-\varepsilon.
\]

Choose now $q$-adic sets $A_q$ such that $|A_q\triangle A|\to0$. Lemma~\ref{lem:l1-continuity} then implies that
\[
\Phi_{k-1}(A_q)\to \Phi_{k-1}(A).
\]
Each $A_q$ is of the form $A_{S_q}$ for some $S_q\subseteq[q]^{k-1}$, and therefore
\[
\frac{M_{k,q}}{q^k}\ge \frac{N_{k,q}(S_q)}{q^k}=\Phi_{k-1}(A_q).
\]
For all sufficiently large $q$,
\[
\frac{M_{k,q}}{q^k}\ge \Phi_{k-1}(A_q)>\Phi_{k-1}(A)-\varepsilon>\lambda_{k-1}-2\varepsilon.
\]
Thus
\[
\liminf_{q\to\infty}\frac{M_{k,q}}{q^k}\ge \lambda_{k-1}-2\varepsilon.
\]
Since $\varepsilon>0$ is arbitrary, we conclude that
\[
\liminf_{q\to\infty}\frac{M_{k,q}}{q^k}\ge \lambda_{k-1}.
\]
Combining this with the limsup bound proves
\[
\lim_{q\to\infty}\frac{M_{k,q}}{q^k}=\lambda_{k-1}.
\]
Finally, Proposition~\ref{prop:exact-reduction} gives
\[
0\le \alpha(k,q)-M_{k,q}\le q,
\]
so dividing by $q^k$ and using $M_{k,q}/q^k\to\lambda_{k-1}$ yields
\[
\frac{\alpha(k,q)}{q^k}\to \lambda_{k-1}.
\]
This proves the stated asymptotic formula.
\end{proof}

\begin{proposition}\label{prop:odd-no-qkm1}
Assume that $k$ is odd and define
\[
\delta(k):=
\begin{cases}
1, & \text{if $k$ is prime},\\[0.3em]
\dfrac{k}{p}, & \text{if $k$ is composite, where $p$ is the smallest prime divisor of $k$}.
\end{cases}
\]
Equivalently, for composite $k$, the quantity $\delta(k)$ is the largest proper divisor of $k$.

For each divisor $s\mid k$, let
\[
\eta_s(q):=\frac{1}{s}\sum_{d\mid s}\mu(d)\,q^{s/d},
\]
the number of cyclic rotation orbits in $[q]^k$ of size $s$. Then
\[
M_{k,q}\le U_k(q):=\frac12\sum_{s\mid k}(s-1)\eta_s(q),
\]
and therefore
\[
M_{k,q}\le \frac{k-1}{2k}q^k+O(q^{\delta(k)}).
\]
Moreover,
\[
M_{k,q}\ge \frac{k-1}{2k}q^k-O(q^{k-2}).
\]
Consequently,
\[
\frac{k-1}{2k}q^k-O(q^{k-2})
\le M_{k,q}\le \frac{k-1}{2k}q^k+O(q^{\delta(k)}),
\]
and, by Proposition~\ref{prop:exact-reduction},
\[
\frac{k-1}{2k}q^k-O(q^{k-2})
\le \alpha(k,q)\le \frac{k-1}{2k}q^k+O(q^{\delta(k)}).
\]
In particular, if $k$ is odd prime, then
\[
M_{k,q}\le \frac{k-1}{2k}(q^k-q),
\qquad
\alpha(k,q)\le \frac{k-1}{2k}q^k+O(q).
\]
\end{proposition}

\begin{proof}
We first prove the upper bound for $M_{k,q}=\alpha_{\Loop}(k,q)$.

Let $\sigma:[q]^k\to[q]^k$ be the cyclic shift
\[
\sigma(x_1,\dots,x_k):=(x_2,\dots,x_k,x_1).
\]
For $x\in[q]^k$, let
\[
\mathcal O(x):=\{x,\sigma x,\dots,\sigma^{s-1}x\},
\qquad s:=|\mathcal O(x)|.
\]
Since $\sigma^k=\mathrm{id}$, every orbit size $s$ divides $k$. The vertices in $\mathcal O(x)$ form a directed cycle of length $s$ in $B(k,q)$; when $s=1$ this is a loop, and when $s\ge 2$ it is the usual directed cycle
\[
x\to \sigma x\to \cdots \to \sigma^{s-1}x\to x.
\]
Because $k$ is odd, every divisor $s\mid k$ is odd. Hence the independence number of the induced subgraph on an orbit of size $s$ in the looped model is exactly
\[
\left\lfloor \frac{s}{2}\right\rfloor=\frac{s-1}{2},
\]
with the case $s=1$ included.

Therefore every looped independent set $J\subseteq [q]^k$ satisfies
\[
|J|
\le
\sum_{s\mid k}\frac{s-1}{2}\,\eta_s(q)
=U_k(q),
\]
where $\eta_s(q)$ denotes the number of $\sigma$-orbits of size $s$. Taking the maximum over all looped independent sets gives
\[
M_{k,q}\le U_k(q).
\]

It remains to identify $\eta_s(q)$. A $\sigma$-orbit of size $s$ is the same thing as a word of length $k$ whose minimal period is $s$, modulo cyclic rotation. Such a word is obtained by repeating $k/s$ times a primitive word of length $s$. The number of primitive words of length $s$ is
\[
\sum_{d\mid s}\mu(d)\,q^{s/d},
\]
and dividing by $s$ gives
\[
\eta_s(q)=\frac{1}{s}\sum_{d\mid s}\mu(d)\,q^{s/d}.
\]

Now isolate the contribution of $s=k$:
\[
U_k(q)
=
\frac{k-1}{2}\eta_k(q)
+
\frac12\sum_{\substack{s\mid k\\ s<k}}(s-1)\eta_s(q).
\]
Since
\[
\eta_k(q)=\frac{1}{k}\sum_{d\mid k}\mu(d)\,q^{k/d}
=\frac{1}{k}q^k+O(q^{\delta(k)}),
\]
and every proper divisor $s<k$ satisfies $s\le \delta(k)$, we obtain
\[
U_k(q)=\frac{k-1}{2k}q^k+O(q^{\delta(k)}).
\]
This proves
\[
M_{k,q}\le \frac{k-1}{2k}q^k+O(q^{\delta(k)}).
\]

For the lower bound we use the same explicit family as before. Let
\[
S^{\mathrm{ev}}_{k,q}\subseteq [q]^{k-1}
\]
be the set of $(u_1,\dots,u_{k-1})$ such that the first occurrence of the maximum value among $u_1,\dots,u_{k-1}$ is in an even position. We claim that a word
\[
x=(x_1,\dots,x_k)\in [q]^k
\]
belongs to $\mathcal I_{k,q}(S^{\mathrm{ev}}_{k,q})$ if and only if the first occurrence of the maximum value among $x_1,\dots,x_k$ is in an even position $t\in\{2,4,\dots,k-1\}$.

Indeed, if the first global maximum occurs at such a position $t$, then in the prefix $(x_1,\dots,x_{k-1})$ the first occurrence of the maximum is also at position $t$, whereas in the suffix $(x_2,\dots,x_k)$ it occurs at position $t-1$; hence the prefix lies in $S^{\mathrm{ev}}_{k,q}$ and the suffix does not. Conversely, if $x\in \mathcal I_{k,q}(S^{\mathrm{ev}}_{k,q})$, then the first occurrence of the maximum in the prefix is even and in the suffix is odd, so the first global maximum cannot occur at position $1$ or $k$ and must therefore occur at an even position $t\in\{2,4,\dots,k-1\}$.

Fix such an even position $t$ and let $M\in\{0,\dots,q-1\}$ be the maximum value. Then the first $t-1$ coordinates may be chosen arbitrarily in $\{0,\dots,M-1\}$, the coordinate $x_t$ must equal $M$, and the remaining $k-t$ coordinates may be chosen arbitrarily in $\{0,\dots,M\}$. Hence the number of such words is
\[
M^{t-1}(M+1)^{k-t}.
\]
Summing over all admissible $t$ and $M$ yields
\[
N_{k,q}(S^{\mathrm{ev}}_{k,q})
=
\sum_{\substack{2\le t\le k-1\\ t\ \mathrm{even}}}\ \sum_{M=0}^{q-1}
M^{t-1}(M+1)^{k-t}.
\]
Expanding each summand,
\[
M^{t-1}(M+1)^{k-t}
=
M^{k-1}+(k-t)M^{k-2}+O(M^{k-3}),
\]
where the implied constant depends only on $k$. Since there are $(k-1)/2$ admissible values of $t$, and
\[
\sum_{\substack{2\le t\le k-1\\ t\ \mathrm{even}}}(k-t)
=
\sum_{j=1}^{(k-1)/2}(k-2j)
=
\frac{(k-1)^2}{4},
\]
we get
\[
N_{k,q}(S^{\mathrm{ev}}_{k,q})
=
\frac{k-1}{2}\sum_{M=0}^{q-1}M^{k-1}
+
\frac{(k-1)^2}{4}\sum_{M=0}^{q-1}M^{k-2}
+
O(q^{k-2}).
\]
By Faulhaber's formula,
\[
\sum_{M=0}^{q-1}M^{k-1}
=
\frac{q^k}{k}-\frac12 q^{k-1}+O(q^{k-2}),
\qquad
\sum_{M=0}^{q-1}M^{k-2}
=
\frac{q^{k-1}}{k-1}-\frac12 q^{k-2}+O(q^{k-3}).
\]
Substituting these expansions gives
\[
N_{k,q}(S^{\mathrm{ev}}_{k,q})
=
\frac{k-1}{2k}q^k
+
\left(-\frac{k-1}{4}+\frac{k-1}{4}\right)q^{k-1}
+
O(q^{k-2})
=
\frac{k-1}{2k}q^k+O(q^{k-2}).
\]
Therefore
\[
M_{k,q}\ge N_{k,q}(S^{\mathrm{ev}}_{k,q})
=
\frac{k-1}{2k}q^k-O(q^{k-2}).
\]

Finally, Proposition~\ref{prop:exact-reduction} yields
\[
M_{k,q}\le \alpha(k,q)\le M_{k,q}+q.
\]
Since $\delta(k)\ge 1$, the upper bound for $M_{k,q}$ implies
\[
\alpha(k,q)\le \frac{k-1}{2k}q^k+O(q^{\delta(k)}),
\]
and the lower bound for $M_{k,q}$ gives
\[
\alpha(k,q)\ge \frac{k-1}{2k}q^k-O(q^{k-2}).
\]
This completes the proof.
\end{proof}


\section{Asymptotic bounds for \texorpdfstring{$k=4$}{k=4}}\label{sec:k4}

For the rest of this section we fix \(k=4\). Accordingly, for every integer \(q\ge 1\) and every set \(S\subseteq[q]^3\) we write
\[
N_q(S):=N_{4,q}(S),
\qquad
\Lambda_q(S):=\Lambda_{4,q}(S),
\qquad
\Phi:=\Phi_3.
\]
Our goal is to prove explicit upper and lower bounds for \(\lambda_3\), and hence for the asymptotic independence ratio of \(B(4,q)\).

\subsection{Upper bound via a seven-cycle inequality}\label{sec:upper}
The upper bound in this section is independent of the dyadic construction and works simultaneously in the discrete and continuum settings.

For a finite set \(S\subseteq[q]^3\) write
\[
\rho_q(S):=\frac{|S|}{q^3}.
\]

\begin{lemma}\label{lem:seven-cycle}
For every binary \(7\)-tuple \(y=(y_i)_{i\in \ZZ/7\ZZ}\in\{0,1\}^{\ZZ/7\ZZ}\) one has
\[
\sum_{i\in \ZZ/7\ZZ} y_i\bigl(1-y_{i+1}\bigr)
\le
1+\sum_{i\in \ZZ/7\ZZ} y_i\bigl(1-y_{i+3}\bigr).
\]
\end{lemma}

\begin{proof}
Let
\[
B:=\{i\in \ZZ/7\ZZ:y_i=1\}.
\]
For \(s\in\{1,3\}\) define
\[
R_s(B):=\#\{i\in B:i+s\notin B\}.
\]
Then the stated inequality is exactly
\[
R_1(B)\le R_3(B)+1.
\]

For every \(s\in\{1,3\}\) we have \(R_s(B)=R_s(B^c)\): indeed, for the permutation \(i\mapsto i+s\) of \(\ZZ/7\ZZ\), the number of transitions \(1\to 0\) equals the number of transitions \(0\to 1\). Hence we may replace \(B\) by its complement and assume
\[
|B|\le 3.
\]

If \(B=\varnothing\), then \(R_1(B)=R_3(B)=0\) and there is nothing to prove. Assume now that \(B\neq\varnothing\). Since \(i\mapsto i+3\) is a single \(7\)-cycle, some element of \(B\) must exit \(B\) under the step \(+3\), so \(R_3(B)\ge 1\). Therefore the claim is immediate whenever \(R_1(B)\le 2\). It remains only to consider the case
\[
R_1(B)=3.
\]
Because \(|B|\le 3\), this forces \(|B|=3\), and the three elements of \(B\) are pairwise nonadjacent in the usual cycle on \(\ZZ/7\ZZ\).

Assume for contradiction that \(R_3(B)=1\). Along the cyclic order generated by repeated addition of \(3\), the indicator of \(B\) changes from \(1\) to \(0\) exactly once, so \(B\) must be a single contiguous block in that \(+3\)-cycle. Since \(|B|=3\), there exists \(t\in\ZZ/7\ZZ\) such that
\[
B=\{t,t+3,t+6\}.
\]
But \(t\) and \(t+6=t-1\) are adjacent in the usual cycle, contradicting \(R_1(B)=3\). Hence \(R_3(B)\ge 2\), and therefore
\[
R_1(B)=3\le 2+1\le R_3(B)+1.
\]
This completes the proof.
\end{proof}

\begin{proposition}\label{prop:discrete-upper}
For every integer \(q\ge 1\) and every set \(S\subseteq[q]^3\),
\[
\Lambda_q(S)\le \rho_q(S)\bigl(1-\rho_q(S)\bigr)+\frac17\le \frac{11}{28}.
\]
Equivalently,
\[
N_q(S)\le \left(\frac{11}{28}\right)q^4.
\]
\end{proposition}

\begin{proof}
For each \(i\in\ZZ/7\ZZ\), define
\[
C_i:=\bigl\{x=(x_0,\dots,x_6)\in[q]^7:(x_i,x_{i+1},x_{i+2})\in S\bigr\},
\]
with indices read modulo \(7\). For a fixed \(x\in[q]^7\), set
\[
y_i:=\1_{C_i}(x)\in\{0,1\}.
\]
Applying Lemma~\ref{lem:seven-cycle} to the binary \(7\)-tuple \((y_i)_{i\in\ZZ/7\ZZ}\) gives
\[
\sum_{i\in\ZZ/7\ZZ} y_i(1-y_{i+1})
\le
1+\sum_{i\in\ZZ/7\ZZ} y_i(1-y_{i+3}).
\]
Since this inequality holds pointwise for every \(x\in[q]^7\), summing over all \(x\) yields
\begin{equation}\label{eq:upper-master-comb}
\sum_{i\in\ZZ/7\ZZ} |C_i\setminus C_{i+1}|
\le
q^7+\sum_{i\in\ZZ/7\ZZ} |C_i\setminus C_{i+3}|.
\end{equation}

We evaluate the two sides of \eqref{eq:upper-master-comb}. By cyclic symmetry, the cardinalities do not depend on \(i\).

For \(C_i\setminus C_{i+1}\), the condition \(x\in C_i\setminus C_{i+1}\) is equivalent to
\[
(x_i,x_{i+1},x_{i+2})\in S,
\qquad
(x_{i+1},x_{i+2},x_{i+3})\notin S,
\]
while the remaining three coordinates are arbitrary. Therefore
\[
|C_i\setminus C_{i+1}|=q^3N_q(S)=q^7\Lambda_q(S).
\]
Hence the left-hand side of \eqref{eq:upper-master-comb} is \(7q^7\Lambda_q(S)\).

For \(C_i\setminus C_{i+3}\), the condition \(x\in C_i\setminus C_{i+3}\) is equivalent to
\[
(x_i,x_{i+1},x_{i+2})\in S,
\qquad
(x_{i+3},x_{i+4},x_{i+5})\notin S,
\]
with \(x_{i+6}\) arbitrary. Here the two displayed constraints involve disjoint coordinate triples, namely \((x_i,x_{i+1},x_{i+2})\) and \((x_{i+3},x_{i+4},x_{i+5})\), while \(x_{i+6}\) is free. Thus
\[
|C_i\setminus C_{i+3}|=|S|\,(q^3-|S|)\,q
=\rho_q(S)\bigl(1-\rho_q(S)\bigr)q^7.
\]
Hence the right-hand side of \eqref{eq:upper-master-comb} is
\[
q^7+7\rho_q(S)\bigl(1-\rho_q(S)\bigr)q^7.
\]
Substituting into \eqref{eq:upper-master-comb} and dividing by \(7q^7\), we obtain
\[
\Lambda_q(S)\le \rho_q(S)\bigl(1-\rho_q(S)\bigr)+\frac17.
\]
Since \(\rho_q(S)(1-\rho_q(S))\le 1/4\), it follows that
\[
\Lambda_q(S)\le \frac14+\frac17=\frac{11}{28}.
\]
Multiplying by \(q^4\) gives the equivalent bound for \(N_q(S)\).
\end{proof}

\begin{corollary}\label{cor:lambda-upper}
One has
\[
\lambda_3\le \frac{11}{28}\,.
\]
\end{corollary}

\begin{proof}
This is the measure-theoretic analogue of Proposition~\ref{prop:discrete-upper}. Let \(A\subseteq[0,1]^3\) be measurable, and write \(|A|:=\Leb^3(A)\). For each \(i\in\ZZ/7\ZZ\), define
\[
C_i:=\bigl\{x=(x_0,\dots,x_6)\in[0,1]^7:(x_i,x_{i+1},x_{i+2})\in A\bigr\},
\]
with indices read modulo \(7\). For a fixed \(x\in[0,1]^7\), set
\[
y_i:=\1_{C_i}(x)\in\{0,1\}.
\]
Applying Lemma~\ref{lem:seven-cycle} to the binary \(7\)-tuple \((y_i)_{i\in\ZZ/7\ZZ}\) gives
\[
\sum_{i\in\ZZ/7\ZZ} y_i(1-y_{i+1})
\le
1+\sum_{i\in\ZZ/7\ZZ} y_i(1-y_{i+3}).
\]
Integrating over \([0,1]^7\) yields
\begin{equation}\label{eq:upper-master-measure}
\sum_{i\in\ZZ/7\ZZ} \Leb^7(C_i\setminus C_{i+1})
\le
1+\sum_{i\in\ZZ/7\ZZ} \Leb^7(C_i\setminus C_{i+3}).
\end{equation}

Again, by cyclic symmetry the measures do not depend on \(i\). For \(C_i\setminus C_{i+1}\),
\[
\Leb^7(C_i\setminus C_{i+1})
=\int_{[0,1]^7} \1_A(x_i,x_{i+1},x_{i+2})\bigl(1-\1_A(x_{i+1},x_{i+2},x_{i+3})\bigr)\,dx.
\]
The integrand depends only on the four coordinates $(x_i,x_{i+1},x_{i+2},x_{i+3})$; after integrating out the remaining three free variables and relabelling
$(u,v,w,z)=(x_i,x_{i+1},x_{i+2},x_{i+3})$, we obtain
\[
\Leb^7(C_i\setminus C_{i+1})
=\int_{[0,1]^4} \1_A(u,v,w)\bigl(1-\1_A(v,w,z)\bigr)\,dudvdwdz
=\Phi(A).
\]
Hence the left-hand side of \eqref{eq:upper-master-measure} is \(7\Phi(A)\).

For \(C_i\setminus C_{i+3}\),
\[
\Leb^7(C_i\setminus C_{i+3})
=\int_{[0,1]^7} \1_A(x_i,x_{i+1},x_{i+2})\bigl(1-\1_A(x_{i+3},x_{i+4},x_{i+5})\bigr)\,dx.
\]
The two factors depend on disjoint triples of coordinates, while \(x_{i+6}\) is free. By Fubini,
\[
\Leb^7(C_i\setminus C_{i+3})
=\left(\int_{[0,1]^3}\1_A\right)\left(\int_{[0,1]^3}(1-\1_A)\right)
=|A|(1-|A|).
\]
Therefore the right-hand side of \eqref{eq:upper-master-measure} is
\[
1+7|A|(1-|A|).
\]
Substituting into \eqref{eq:upper-master-measure}, we find
\[
\Phi(A)\le |A|(1-|A|)+\frac17\le \frac14+\frac17=\frac{11}{28}.
\]
Taking the supremum over measurable \(A\subseteq[0,1]^3\) proves the claim.
\end{proof}

\subsection{Lower bound via a seven-site gadget}\label{sec:lower}
We now construct an explicit dyadic family of sets that yields the lower bound \(\lambda_3\ge 91/240\).

Given \(S\subseteq[q]^3\), its \emph{dyadic lift} \(L_q(S)\subseteq[2q]^3\) is defined by
\[
(r,s,t)\in L_q(S)
\iff
\bigl(\lfloor r/2\rfloor,\lfloor s/2\rfloor,\lfloor t/2\rfloor\bigr)\in S.
\]

\begin{lemma}\label{lem:dyadic-lift}
For every \(q\ge 1\) and every \(S\subseteq[q]^3\),
\[
N_{2q}(L_q(S))=16\,N_q(S).
\]
\end{lemma}

\begin{proof}
A quadruple \((\alpha,\beta,\gamma,\delta)\in[q]^4\) contributes to \(N_q(S)\) if and only if
\[
(\alpha,\beta,\gamma)\in S,
\qquad
(\beta,\gamma,\delta)\notin S.
\]
A quadruple \((r,s,t,u)\in[2q]^4\) contributes to \(N_{2q}(L_q(S))\) if and only if its floor image
\[
(\lfloor r/2\rfloor,\lfloor s/2\rfloor,\lfloor t/2\rfloor,\lfloor u/2\rfloor)
\]
contributes to \(N_q(S)\); membership and non-membership are preserved under the floor map by definition of the lift. Each contributing quadruple in \([q]^4\) has exactly \(2^4=16\) lifts to \([2q]^4\). Hence
\[
N_{2q}(L_q(S))=16\,N_q(S).
\]
\end{proof}

The next ingredient is a local seven-site modification that increases the count by exactly one while preserving the local hypothesis needed for iteration.

Let \(a,b\in[q]\) with \(a\neq b\). We say that \(H_q(a,b;S)\) holds if
\begin{align*}
&\text{(i) }(a,a,a)\in S,\quad (b,b,a)\in S,\quad (b,a,a)\notin S,\quad (b,b,b)\notin S,\\
&\text{(ii) }I_S(a,a)+O_S(a,a)=q,\\
&\text{(iii) }I_S(b,b)+O_S(b,b)=q,\\
&\text{(iv) }I_S(b,b)+O_S(b,a)=q-1,\\
&\text{(v) }I_S(b,a)+O_S(a,a)=q+1.
\end{align*}

Assume now that \(H_q(a,b;S)\) holds, and set \(X:=L_q(S)\subseteq[2q]^3\). We define seven distinguished sites in \([2q]^3\):
\begin{align*}
\text{Additions:}&\quad P_1=(2b,2b,2b+1),\quad P_2=(2b+1,2b,2b+1),\quad P_3=(2b+1,2a+1,2a),\\
\text{Removals:}&\quad M_1=(2a+1,2a,2a),\quad M_2=(2a+1,2a,2a+1),\\
&\quad M_3=(2b,2b+1,2a),\quad M_4=(2b,2b+1,2a+1).
\end{align*}
All seven sites are pairwise distinct because \(a\neq b\). Moreover, the membership part of \(H_q(a,b;S)\) implies that
\begin{itemize}
\item[--] the microcell above \((b,b,b)\) is absent from \(X\), so \(P_1,P_2\notin X\);
\item[--]  the microcell above \((b,a,a)\) is absent from \(X\), so \(P_3\notin X\);
\item[--]  the microcell above \((a,a,a)\) is present in \(X\), so \(M_1,M_2\in X\);
\item[--]  the microcell above \((b,b,a)\) is present in \(X\), so \(M_3,M_4\in X\);
\end{itemize}
hence the operation is 
\[
T := \bigl(X\cup\{P_1,P_2,P_3\}\bigr)\setminus\{M_1,M_2,M_3,M_4\}.
\]

\begin{lemma}\label{lem:gadget-k4}
Under the assumptions above, we have
\[
N_{2q}(T)=16\,N_q(S)+1.
\]
\end{lemma}

\begin{proof}
By Lemma~\ref{lem:dyadic-lift},
\[
N_{2q}(X)=16N_q(S).
\]
We compute the increment
\[
\Delta N:=N_{2q}(T)-N_{2q}(X)
\]
using the degree formula \eqref{eq:Nsum-general}. For \((u,v)\in[2q]^2\), set
\[
\Delta I(u,v):=I_T(u,v)-I_X(u,v),
\qquad
\Delta O(u,v):=O_T(u,v)-O_X(u,v).
\]
Expanding
\[
(I_X+\Delta I)\bigl(2q-(O_X+\Delta O)\bigr)-I_X(2q-O_X)
\]
inside the sum \eqref{eq:Nsum-general} gives
\[
\Delta N = \sum_{u,v}\Bigl[\Delta I(u,v)\bigl(2q-O_X(u,v)\bigr)-I_X(u,v)\Delta O(u,v)-\Delta I(u,v)\Delta O(u,v)\Bigr].
\]

Only finitely many ordered pairs \((u,v)\) are affected. For each such pair we also record its coarse image \((\lfloor u/2\rfloor,\lfloor v/2\rfloor)\). Since \(X=L_q(S)\), we have
\[
I_X(u,v)=2I_S(\lfloor u/2\rfloor,\lfloor v/2\rfloor),
\qquad
O_X(u,v)=2O_S(\lfloor u/2\rfloor,\lfloor v/2\rfloor).
\]
We first localise the effect of the seven toggled triples on the ordered pairs $(u,v)$, and then rewrite the increment $\Delta N$ through a small collection of scalar balance quantities. A direct inspection yields the following variation table:
\[
\begin{array}{c|c|c|c}
(u,v) & (\lfloor u/2\rfloor,\lfloor v/2\rfloor) & \Delta O & \Delta I \\ \hline
(2b,2b)      & (b,b) & +1 & 0 \\
(2b+1,2b)    & (b,b) & +1 & 0 \\
(2b+1,2a+1)  & (b,a) & +1 & -1 \\
(2a+1,2a)    & (a,a) & -2 & +1 \\
(2a,2a)      & (a,a) & 0  & -1 \\
(2a,2a+1)    & (a,a) & 0  & -1 \\
(2b,2b+1)    & (b,b) & -2 & +2 \\
(2b+1,2a)    & (b,a) & 0  & -1
\end{array}
\]
No other ordered pair is affected, hence every other pair contributes \(0\).

For brevity write
\[
\begin{aligned}
I_{aa}&:=I_S(a,a), & O_{aa}&:=O_S(a,a), & I_{bb}&:=I_S(b,b),\\
O_{bb}&:=O_S(b,b), & I_{ba}&:=I_S(b,a), & O_{ba}&:=O_S(b,a).
\end{aligned}
\]
Using the table, the eight non-zero contributions to \(\Delta N\) are:
\begin{align*}
(2b,2b):&\quad -2I_{bb},\\
(2b+1,2b):&\quad -2I_{bb},\\
(2b+1,2a+1):&\quad -2q+2O_{ba}-2I_{ba}+1,\\
(2a+1,2a):&\quad 2q-2O_{aa}+4I_{aa}+2,\\
(2a,2a):&\quad -2q+2O_{aa},\\
(2a,2a+1):&\quad -2q+2O_{aa},\\
(2b,2b+1):&\quad 4q-4O_{bb}+4I_{bb}+4,\\
(2b+1,2a):&\quad -2q+2O_{ba}.
\end{align*}
Summing them gives
\[
\Delta N=-2I_{ba}+4I_{aa}+2O_{aa}+4O_{ba}-4O_{bb}-2q+7.
\]

Now invoke the four degree relations contained in \(H_q(a,b;S)\):
\[
I_{aa}+O_{aa}=q,
\qquad
I_{bb}+O_{bb}=q,
\qquad
I_{bb}+O_{ba}=q-1,
\qquad
I_{ba}+O_{aa}=q+1.
\]
From these we obtain
\[
O_{aa}=q+1-I_{ba},
\qquad
I_{aa}=I_{ba}-1,
\qquad
O_{ba}=q-1-I_{bb},
\qquad
O_{bb}=q-I_{bb}.
\]
Substituting into the formula for \(\Delta N\),
\begin{align*}
\Delta N
&= -2I_{ba}+4(I_{ba}-1)+2(q+1-I_{ba})+4(q-1-I_{bb})-4(q-I_{bb})-2q+7\\
&=1.
\end{align*}
Therefore
\[
N_{2q}(T)=N_{2q}(X)+1=16N_q(S)+1,
\]
as claimed.
\end{proof}

\subsubsection{Propagation of the local hypothesis}

For \(u=(u_1,u_2,u_3)\in[q]^3\), write
\[
B(u):=\{(2u_1+\varepsilon_1,2u_2+\varepsilon_2,2u_3+\varepsilon_3): \varepsilon_1,\varepsilon_2,\varepsilon_3\in\{0,1\}\}\subseteq[2q]^3.
\]
We call \(B(u)\) the \emph{microcell above \(u\)}.

\begin{proposition}\label{prop:propagation-k4}
Let \(S\subseteq[q]^3\), and let \(a,b\in[q]\) with \(a\neq b\). Assume that \(H_q(a,b;S)\) holds. Let
\[
X:=L_q(S)\subseteq[2q]^3,
\]
and let \(T\subseteq[2q]^3\) be obtained from \(X\) by the seven-site gadget above. Define
\[
a':=2a+1,
\qquad
b':=2b+1.
\]
Then
\[
H_{2q}(a',b';T)
\]
holds.
\end{proposition}

\begin{proof}
All seven toggled sites lie in the four microcells above \((a,a,a)\), \((b,b,a)\), \((b,a,a)\), and \((b,b,b)\). We verify the four membership conditions and the four degree-balance conditions for the new pair \((a',b')=(2a+1,2b+1)\).

\medskip
\noindent\textbf{1. Membership conditions.}

Because \((a,a,a)\in S\), the whole microcell above \((a,a,a)\) lies in \(X\), hence
\[
(a',a',a')=(2a+1,2a+1,2a+1)\in X.
\]
This site is not toggled, so \((a',a',a')\in T\).

Because \((b,b,a)\in S\), the whole microcell above \((b,b,a)\) lies in \(X\), hence
\[
(b',b',a')=(2b+1,2b+1,2a+1)\in X\subseteq T,
\]
again because this site is not toggled.

Because \((b,a,a)\notin S\), the whole microcell above \((b,a,a)\) is absent from \(X\), so
\[
(b',a',a')=(2b+1,2a+1,2a+1)\notin X.
\]
Among the three additions, only \(P_3=(2b+1,2a+1,2a)\) lies in that microcell, so \((b',a',a')\) is still absent from \(T\).

Because \((b,b,b)\notin S\), the whole microcell above \((b,b,b)\) is absent from \(X\), in particular
\[
(b',b',b')=(2b+1,2b+1,2b+1)\notin X.
\]
The gadget adds only \(P_1\) and \(P_2\) in that microcell, so \((b',b',b')\notin T\).

Thus the membership part of \(H_{2q}(a',b';T)\) holds.

\medskip
\noindent\textbf{2. The identity \(I_T(a',a')+O_T(a',a')=2q\).}

No toggled site has initial pair \((a',a')\), and no toggled site has terminal pair \((a',a')\). Therefore
\[
I_T(a',a')=I_X(a',a')=2I_S(a,a),
\qquad
O_T(a',a')=O_X(a',a')=2O_S(a,a).
\]
Using \(I_S(a,a)+O_S(a,a)=q\), we obtain
\[
I_T(a',a')+O_T(a',a')=2q.
\]

\medskip
\noindent\textbf{3. The identity \(I_T(b',b')+O_T(b',b')=2q\).}

Exactly the same argument gives
\[
I_T(b',b')=2I_S(b,b),
\qquad
O_T(b',b')=2O_S(b,b),
\]
so from \(I_S(b,b)+O_S(b,b)=q\) we get
\[
I_T(b',b')+O_T(b',b')=2q.
\]

\medskip
\noindent\textbf{4. The identity \(I_T(b',b')+O_T(b',a')=2q-1\).}

We already know that \(I_T(b',b')=2I_S(b,b)\). Moreover, \(O_T(b',a')\) counts triples of \(T\) whose first two coordinates are \((b',a')=(2b+1,2a+1)\). In the lift \(X\), this count is \(2O_S(b,a)\). The gadget adds exactly one such triple, namely
\[
P_3=(2b+1,2a+1,2a),
\]
and removes none with the same initial pair. Hence
\[
O_T(b',a')=2O_S(b,a)+1.
\]
Using \(I_S(b,b)+O_S(b,a)=q-1\), we conclude that
\[
I_T(b',b')+O_T(b',a')
=2I_S(b,b)+2O_S(b,a)+1
=2q-1.
\]

\medskip
\noindent\textbf{5. The identity \(I_T(b',a')+O_T(a',a')=2q+1\).}

We already know that \(O_T(a',a')=2O_S(a,a)\). On the other hand, \(I_T(b',a')\) counts triples of \(T\) whose last two coordinates are \((b',a')=(2b+1,2a+1)\). In the lift \(X\), this count is \(2I_S(b,a)\). The gadget removes exactly one such triple, namely
\[
M_4=(2b,2b+1,2a+1),
\]
and adds none with the same terminal pair. Therefore
\[
I_T(b',a')=2I_S(b,a)-1.
\]
Using \(I_S(b,a)+O_S(a,a)=q+1\), we obtain
\[
I_T(b',a')+O_T(a',a')
=2I_S(b,a)-1+2O_S(a,a)
=2q+1.
\]

All four membership conditions and all four degree-balance conditions hold, so \(H_{2q}(a',b';T)\) follows.
\end{proof}

\subsubsection{The explicit base configuration at scale \texorpdfstring{$q=16$}{q=16}}\label{subsec:base-q16}
The lower-bound construction requires one explicit seed configuration $S_{16}\subseteq[16]^3$ satisfying a local compatibility hypothesis and having a sufficiently large value of $N_{16}$. The full specification of $S_{16}$ is finite but bulky: it is described by ten fibres and a $16\times 16$ fibre table. Since these data play a purely foundational role in the induction, we place the complete definition in Appendix~\ref{app:q16-data} and keep only the summary needed for the argument here.

More precisely, Appendix~\ref{app:q16-data} defines ten fibres
\[
\cA,\cB,\cC,\cD,\cE,\cF,\cG,\cHh,\cI,\cJ\subseteq[16]
\]
and a fibre table $(T_{a,b})_{(a,b)\in[16]^2}$ with values among these ten sets. The seed configuration is then
\[
(a,b,c)\in S_{16}
\quad\Longleftrightarrow\quad
c\in T_{a,b}.
\]
The choice $q=16=2^4$ is the natural base scale for the construction: the propagation step is dyadic, so once a single admissible seed is available at one power of two, it propagates canonically to all larger dyadic scales. We have not attempted to optimise the smallest possible seed size; rather, $q=16$ is the first scale at which we found a convenient explicit configuration with enough room to enforce the local hypothesis and achieve the target density.

The next lemma records exactly the two facts from this explicit dataset that are needed later: the base count and the local hypothesis required by the dyadic propagation. Their complete finite verification is given in Appendix~\ref{app:q16}.

\begin{lemma}\label{lem:base-k4-q16}
Let $S_{16}\subseteq [16]^3$ be the explicit set defined in Appendix~\ref{app:q16-data}. Then:
\begin{enumerate}
\item
\[
N_{16}(S_{16})=24849.
\]
\item The local hypothesis $H_{16}(8,14;S_{16})$ holds. Equivalently,
\begin{align*}
&(8,8,8)\in S_{16},\qquad (14,14,8)\in S_{16},\qquad (14,8,8)\notin S_{16},\qquad (14,14,14)\notin S_{16},\\
&I_{S_{16}}(8,8)+O_{S_{16}}(8,8)=16,\\
&I_{S_{16}}(14,14)+O_{S_{16}}(14,14)=16,\\
&I_{S_{16}}(14,14)+O_{S_{16}}(14,8)=15,\\
&I_{S_{16}}(14,8)+O_{S_{16}}(8,8)=17.
\end{align*}
\end{enumerate}
\end{lemma}

\begin{proof}
The full finite verification is recorded in Appendix~\ref{app:q16}. In brief, the four membership assertions are read directly from the fibre table, the four degree identities are obtained by counting the occurrences of the values $8$ and $14$ in the corresponding rows and columns, and the counting identity
\[
N_{16}(S_{16})=
\sum_{u,v=0}^{15} I_{S_{16}}(u,v)\bigl(16-O_{S_{16}}(u,v)\bigr)
\]
is then evaluated from the complete incoming- and outgoing-degree matrices. The resulting row sums add up to $24849$.
\end{proof}

\subsubsection{Inductive construction and the resulting lower bound}\label{sec:lower-construction}

\begin{proposition}\label{prop:dyadic-k4}
Set \(q_m:=2^m\) for \(m\ge 4\). Define
\[
S_{q_4}:=S_{16},
\qquad
a_4:=8,
\qquad
b_4:=14,
\]
and recursively, for \(m\ge 4\), let \(S_{q_{m+1}}\subseteq[q_{m+1}]^3\) be obtained from \(S_{q_m}\) by first taking the dyadic lift and then applying the seven-site gadget with parameters \((a_m,b_m)\). Set
\[
a_{m+1}:=2a_m+1,
\qquad
b_{m+1}:=2b_m+1.
\]
Then for every \(m\ge 4\):
\begin{enumerate}
\item \(H_{q_m}(a_m,b_m;S_{q_m})\) holds;
\item
\[
N_{q_{m+1}}(S_{q_{m+1}})=16\,N_{q_m}(S_{q_m})+1;
\]
\item
\[
N_{q_m}(S_{q_m})=\frac{91}{240}\,q_m^4-\frac1{15}.
\]
\end{enumerate}
\end{proposition}

\begin{proof}
The base lemma gives \(H_{16}(8,14;S_{16})\), so the claim in (1) holds for \(m=4\). If \(H_{q_m}(a_m,b_m;S_{q_m})\) holds, then Proposition~\ref{prop:propagation-k4} gives
\[
H_{q_{m+1}}(a_{m+1},b_{m+1};S_{q_{m+1}}),
\]
which proves (1) for all \(m\ge 4\) by induction.

Under the same inductive hypothesis, Lemma~\ref{lem:gadget-k4} yields
\[
N_{q_{m+1}}(S_{q_{m+1}})=16\,N_{q_m}(S_{q_m})+1,
\]
which is (2).

To solve the recurrence, use the initial value \(N_{16}(S_{16})=24849\). Iterating (2) gives
\[
N_{q_m}(S_{q_m})
=16^{m-4}\cdot 24849+\sum_{j=0}^{m-5}16^j
=16^{m-4}\cdot 24849+\frac{16^{m-4}-1}{15}.
\]
Since \(q_m=2^m\), so that \(16^{m-4}=q_m^4/16^4=q_m^4/65536\), and since
\[
\frac{24849}{65536}+\frac{1}{15\cdot 65536}=\frac{91}{240},
\]
we obtain
\[
N_{q_m}(S_{q_m})=\frac{91}{240}\,q_m^4-\frac1{15}.
\]
This proves (3).
\end{proof}

\begin{corollary}\label{cor:k4-combined-bounds}
\[
\frac38<\frac{91}{240}\le \lambda_3\le \frac{11}{28}<\frac25.
\]
\end{corollary}

\begin{proof}
For each \(m\ge 4\), let \(A_m:=A_{S_{q_m}}\subseteq[0,1]^3\) be the \(q_m\)-adic set associated with \(S_{q_m}\). By Proposition~\ref{prop:dyadic-k4} and the correspondence lemma,
\[
\Phi(A_m)=\Lambda_{q_m}(S_{q_m})=\frac{N_{q_m}(S_{q_m})}{q_m^4}=\frac{91}{240}-\frac{1}{15q_m^4}.
\]
Letting \(m\to\infty\) gives
\[
\lambda_3\ge \frac{91}{240}.
\]
The upper bound \(\lambda_3\le 11/28\) is Corollary~\ref{cor:lambda-upper}. The strict inequalities with \(3/8\) and \(2/5\) are immediate.
\end{proof}

\subsection{The independence number of $B(4,q)$}\label{sec:consequences}
We now combine Proposition~\ref{prop:exact-reduction}, the asymptotic result of Section~\ref{sec:asymptotic}, and the explicit bounds for \(\lambda_3\) established above.

\begin{theorem}\label{thm:main-db4}
As \(q\to\infty\),
\[
\alpha(4,q)=\lambda_3 q^4+o(q^4),
\]
with
\[
\frac{91}{240}\le \lambda_3\le \frac{11}{28}.
\]
Equivalently,
\[
\frac{91}{240}q^4+o(q^4)\le \alpha(4,q)\le \frac{11}{28}q^4+o(q^4).
\]
The same asymptotic estimate also holds for \(\alpha_{\Loop}(4,q)\).
\end{theorem}

\begin{proof}
The asymptotic formula is exactly Theorem~\ref{thm:limit-exists} specialised to \(k=4\). The upper bound for \(\lambda_3\) is Corollary~\ref{cor:lambda-upper}, and the lower bound is Corollary~\ref{cor:k4-combined-bounds}. Substituting these bounds into Theorem~\ref{thm:limit-exists} yields the claim.
\end{proof}

Along the dyadic subsequence \(q=2^m\), Proposition~\ref{prop:dyadic-k4} gives the sharper explicit bound
\[
N_{2^m}(S_{2^m})=\frac{91}{240}\,2^{4m}-\frac1{15}.
\]
Therefore
\[
\alpha_{\Loop}(4,2^m)\ge \frac{91}{240}\,2^{4m}-\frac1{15},
\qquad
\alpha(4,2^m)\ge \frac{91}{240}\,2^{4m}-\frac1{15}.
\]
The universal upper bound from Proposition~\ref{prop:discrete-upper} reads
\[
\alpha_{\Loop}(4,q)\le \frac{11}{28}q^4,
\qquad
\alpha(4,q)\le \frac{11}{28}q^4+q.
\]

\begin{remark}[Road map of the $k=4$ lower bound]
The lower bound has four logically distinct layers. One starts from the explicit seed $S_{16}$, verifies finitely that it has the required local structure and that $N_{16}(S_{16})=24849$, propagates this information to all dyadic scales by the lift-plus-gadget construction, and finally passes from dyadic discrete sets to the continuum variational problem through the correspondence established in Section~\ref{sec:asymptotic}. This separation is useful for refereeing: only the first layer is finite bookkeeping, while the later steps are conceptual and remain unchanged once the base seed has been verified.
\end{remark}

\section{Exact formulas for \texorpdfstring{$k=11,13,17$}{k=11,13,17}}\label{sec:odd-prime-binary}
In this section we specialise to the binary graph underlying $B(k,2)$, where $k$ is an odd prime. We write words as
\[
x=x_0x_1\dots x_{k-1}\in\{0,1\}^k,
\]
and let
\[
\rho(x_0x_1\dots x_{k-1}):=x_1x_2\dots x_{k-1}x_0
\]
be the cyclic left rotation. Throughout this section, orbit indices are understood modulo $k$. We also write $0^k$ and $1^k$ for the two constant words. This is the only place in the paper where we switch to $0$-based coordinate indexing, in order to align the notation with the cyclic rotation action.

Two results of Lichiardopol are the external input for what follows. First, Propositions~4.2 and~4.3(b) of \cite{Lich06}, specialised to the odd prime binary case, give
\begin{equation}\label{eq:binary-upper-prime}
\alpha(k,2)\le 1+\frac{k-1}{2k}(2^k-2).
\end{equation}
Second, by \cite[Theorem~4.4]{Lich06}, if the bound \eqref{eq:binary-upper-prime} is attained by an independent set containing exactly one loop-vertex, then for every alphabet size $q\ge2$ one has
\begin{equation}\label{eq:lich-lift-prime}
\alpha(k,q)=\frac{(k-1)(q^k-q)}{2k}+1,
\qquad
\alpha_{\Loop}(k,q)=\frac{(k-1)(q^k-q)}{2k}.
\end{equation}
Thus the binary problem controls the exact formulas for all $q\ge 2$.

\subsection{Rotation orbits and the phase reduction}
For $x\in\{0,1\}^k$, let
\[
C(x):=\{\rho^i(x):i\in\ZZ/k\ZZ\}
\]
be its rotation orbit.

\begin{lemma}\label{lem:orbit-size-prime}
If $x\in\{0,1\}^k$ is non-constant, then $|C(x)|=k$.
\end{lemma}

\begin{proof}
If $\rho^r(x)=x$ for some $r\in\ZZ/k\ZZ$, then $x$ has cyclic period $r$. Since $k$ is prime, every non-zero class in $\ZZ/k\ZZ$ generates the whole group, so a non-trivial period would force all coordinates of $x$ to coincide. Hence a non-constant word has no non-trivial stabiliser under rotation, and its orbit has cardinality $k$.
\end{proof}

\begin{lemma}\label{lem:orbit-cycle-prime}
Let $x\in\{0,1\}^k$ be non-constant, and write $x^{(i)}:=\rho^i(x)$ for $i\in\ZZ/k\ZZ$. Then the subgraph induced by $C(x)$ in the simple graph underlying $B(k,2)$ is the cycle graph $C_k$. Equivalently,
\[
x^{(i)}\sim x^{(j)}
\qquad\Longleftrightarrow\qquad
j-i\equiv \pm1 \pmod k.
\]
\end{lemma}

\begin{proof}
Because adjacency is invariant under simultaneous rotation of both words, it is enough to determine when $x^{(0)}$ is adjacent to $x^{(r)}$. If the suffix of length $k-1$ of $x^{(0)}$ equals the prefix of length $k-1$ of $x^{(r)}$, then $x$ has cyclic period $r-1$; if the prefix of length $k-1$ of $x^{(0)}$ equals the suffix of length $k-1$ of $x^{(r)}$, then $x$ has cyclic period $r+1$. Since $x$ is non-constant and $k$ is prime, Lemma~\ref{lem:orbit-size-prime} implies that only the trivial period can occur. Thus $r\equiv1$ or $r\equiv-1\pmod k$, and the claim follows.
\end{proof}

Let
\[
J_k:=\{1,3,5,\dots,k-2\}\subseteq \ZZ/k\ZZ.
\]
For a non-trivial rotation orbit $C=(x^{(i)})_{i\in\ZZ/k\ZZ}$ and a phase $t\in\ZZ/k\ZZ$, define
\[
A_t(C):=\{x^{(t+r)}:r\in J_k\}.
\]

\begin{lemma}\label{lem:alternating-phase-set}
For every non-trivial rotation orbit $C$ and every $t\in\ZZ/k\ZZ$, the set $A_t(C)$ is independent and has cardinality $(k-1)/2$.
\end{lemma}

\begin{proof}
By Lemma~\ref{lem:orbit-cycle-prime}, the only edges inside $C$ join consecutive residues modulo $k$. The translate $J_k+t$ contains no two consecutive residues, hence $A_t(C)$ is independent. Its cardinality is $|J_k|=(k-1)/2$.
\end{proof}

\begin{proposition}\label{prop:phase-reduction-prime}
Let $S\subseteq\{0,1\}^k$ be an independent set such that
\[
|S|=1+\frac{k-1}{2k}(2^k-2)
\]
and such that $S$ contains exactly one loop-vertex. Then for every non-trivial rotation orbit $C$ one has
\[
|S\cap C|=\frac{k-1}{2},
\]
and there exists a unique phase $t_C\in\ZZ/k\ZZ$ such that
\[
S\cap C=A_{t_C}(C).
\]
\end{proposition}

\begin{proof}
The $2^k-2$ non-constant words split into
\[
N_k:=\frac{2^k-2}{k}
\]
non-trivial rotation orbits by Lemma~\ref{lem:orbit-size-prime}. Since $S$ contains exactly one loop-vertex and has cardinality
\[
1+\frac{k-1}{2}N_k,
\]
its non-constant part has size exactly $\frac{k-1}{2}N_k$. By Lemma~\ref{lem:orbit-cycle-prime}, each non-trivial orbit induces a copy of $C_k$, so it contributes at most $(k-1)/2$ vertices to $S$. Therefore every non-trivial orbit must attain this maximum. The $k$ maximum independent sets of $C_k$ are exactly the translates of the alternating pattern $J_k$. Hence there exists a unique phase $t_C\in \ZZ/k\ZZ$ such that $S\cap C=A_{t_C}(C)$.
\end{proof}

To express the compatibility constraints between different orbits, let $C=(x^{(i)})_{i\in\ZZ/k\ZZ}$ and $C'=(y^{(j)})_{j\in\ZZ/k\ZZ}$ be two non-trivial rotation orbits and define the adjacency-difference set
\[
D(C,C'):=\{j-i\pmod k:x^{(i)}\sim y^{(j)}\}\subseteq\ZZ/k\ZZ.
\]
The corresponding forbidden phase-difference set is
\[
F(C,C'):=\{\delta-(b-a)\pmod k: \delta\in D(C,C'),\ a,b\in J_k\}.
\]
If $\ell\in\{0^k,1^k\}$ is a chosen loop-vertex and $C=(x^{(i)})_{i\in\ZZ/k\ZZ}$ is a non-trivial orbit, define also
\[
D(\ell,C):=\{i\in\ZZ/k\ZZ: \ell\sim x^{(i)}\}
\]
and
\[
F(\ell,C):=\{\delta-a\pmod k: \delta\in D(\ell,C),\ a\in J_k\}.
\]
Thus a phase $t_C$ is compatible with $\ell$ precisely when $t_C\notin F(\ell,C)$.

\begin{theorem}\label{thm:phase-constraints}
Fix $\ell\in\{0^k,1^k\}$. For each non-trivial rotation orbit $C$, choose a phase $t_C\in\ZZ/k\ZZ$ and set
\[
S_{\ell}(\{t_C\}):=\{\ell\}\cup\bigcup_C A_{t_C}(C),
\]
where the union runs over all non-trivial rotation orbits. Then $S_{\ell}(\{t_C\})$ is independent if and only if the following two conditions hold:
\begin{itemize}
\item[(i)] for every non-trivial orbit $C$, one has $t_C\notin F(\ell,C)$;
\item[(ii)] for every pair of distinct non-trivial orbits $C,C'$, one has
\[
t_{C'}-t_C\notin F(C,C').
\]
\end{itemize}
Whenever these conditions are satisfied,
\[
|S_{\ell}(\{t_C\})|=1+\frac{k-1}{2k}(2^k-2).
\]
Conversely, every independent set of this cardinality with exactly one loop-vertex arises in this way.
\end{theorem}

\begin{proof}
Lemma~\ref{lem:alternating-phase-set} shows that each set $A_{t_C}(C)$ is independent inside its own orbit. A cross-conflict between $A_{t_C}(C)$ and $A_{t_{C'}}(C')$ occurs if and only if there exist $a,b\in J_k$ such that $x^{(t_C+a)}\sim y^{(t_{C'}+b)}$, namely if and only if
\[
(t_{C'}+b)-(t_C+a)\in D(C,C').
\]
This is equivalent to $t_{C'}-t_C\in F(C,C')$, which gives condition (ii). Likewise, $\ell$ is adjacent to some vertex of $A_{t_C}(C)$ if and only if there exist $a\in J_k$ and $\delta\in D(\ell,C)$ such that $t_C+a\equiv\delta\pmod k$, namely if and only if $t_C\in F(\ell,C)$, which is condition (i). The cardinality formula follows from Lemma~\ref{lem:alternating-phase-set}, and the converse is exactly Proposition~\ref{prop:phase-reduction-prime}.
\end{proof}

\begin{remark}\label{rem:sparse-csp-prime}
The phase reduction converts the search for an extremal binary example with one prescribed loop-vertex into a finite difference-CSP on the non-trivial rotation orbits. This CSP is sparse for two elementary reasons. First, if $x\sim y$ in the simple graph underlying $B(k,2)$, then the Hamming weights satisfy $|\mathrm{wt}(x)-\mathrm{wt}(y)|\le1$. Hence only adjacent Hamming-weight layers can interact. Second, each binary word has at most four de Bruijn neighbours, so the orbit-conflict graph on the non-trivial rotation orbits has maximum degree at most $4k$.
\end{remark}

\subsection{Certificates and the cases \texorpdfstring{$k=11,13,17$}{k=11,13,17}}
For the prime cases $k=11$, $k=13$, and $k=17$, we use a single \emph{compact orbit-marker certificate}: it records one marked word for each non-trivial rotation orbit. From this data one reconstructs the entire candidate independent set. The three compact certificates and the verifier script are given in Appendix~\ref{app:certificates}.

The role of the certificates is entirely finite and explicit. Starting from the compact orbit-marker file, one decodes each listed decimal marker as a length-$k$ binary word and treats it as the first selected vertex after the unique double gap in its rotation orbit, selects every second rotation from it, takes the union together with $0^k$, checks the target cardinality, verifies independence directly by testing the at most four de Bruijn neighbours of each selected vertex, and confirms that $1^k$ is absent. Thus the computer-assisted part of Theorems~\ref{thm:k11-binary}, \ref{thm:k13-binary}, and \ref{thm:k17-binary} is a verification statement, not a search statement. For reproducibility, Appendix~\ref{app:certificates} records the compact certificate format, the three compact certificates, and the verifier script. The verifier uses only the Python standard library; no external solver, random search, or hidden preprocessing is required.

\begin{theorem}\label{thm:k11-binary}
One has
\[
\alpha(11,2)=931.
\]
More precisely, there exists an independent set in the simple graph underlying $B(11,2)$ of cardinality $931$ that contains $0^{11}$ and excludes $1^{11}$.
\end{theorem}

\begin{proof}
There are
\[
N_{11}=\frac{2^{11}-2}{11}=186
\]
non-trivial rotation orbits, so the binary upper bound \eqref{eq:binary-upper-prime} gives
\[
\alpha(11,2)\le 1+\frac{10}{22}(2^{11}-2)=1+5\cdot 186=931.
\]
The certificate for $k=11$ records one decimal marker for each of the $186$ non-trivial rotation orbits. Applying the verification protocol described above reconstructs an independent set of cardinality $931$ containing $0^{11}$ and excluding $1^{11}$. Therefore the upper bound is attained, and $\alpha(11,2)=931$.
\end{proof}

\begin{corollary}\label{cor:k11-all-q}
For every $q\ge2$,
\[
\alpha(11,q)=\frac{5(q^{11}-q)}{11}+1,
\qquad
\alpha_{\Loop}(11,q)=\frac{5(q^{11}-q)}{11}.
\]
\end{corollary}

\begin{proof}
The verified binary extremal set in Theorem~\ref{thm:k11-binary} contains exactly one loop-vertex. The formulas follow from the lifting statement \eqref{eq:lich-lift-prime}.
\end{proof}

\begin{theorem}\label{thm:k13-binary}
One has
\[
\alpha(13,2)=3781.
\]
More precisely, there exists an independent set in the simple graph underlying $B(13,2)$ of cardinality $3781$ that contains $0^{13}$ and excludes $1^{13}$.
\end{theorem}

\begin{proof}
There are
\[
N_{13}=\frac{2^{13}-2}{13}=630
\]
non-trivial rotation orbits, so the binary upper bound \eqref{eq:binary-upper-prime} gives
\[
\alpha(13,2)\le 1+\frac{12}{26}(2^{13}-2)=1+6\cdot 630=3781.
\]
The certificate for $k=13$ records one decimal marker for each of the $630$ non-trivial rotation orbits. Applying the verification protocol described above reconstructs an independent set of cardinality $3781$ containing $0^{13}$ and excluding $1^{13}$. Therefore the upper bound is sharp, and $\alpha(13,2)=3781$.
\end{proof}

\begin{corollary}\label{cor:k13-all-q}
For every $q\ge2$,
\[
\alpha(13,q)=\frac{6(q^{13}-q)}{13}+1,
\qquad
\alpha_{\Loop}(13,q)=\frac{6(q^{13}-q)}{13}.
\]
\end{corollary}

\begin{proof}
The verified binary extremal set in Theorem~\ref{thm:k13-binary} contains exactly one loop-vertex. The formulas follow from the lifting statement \eqref{eq:lich-lift-prime}.
\end{proof}

\begin{theorem}\label{thm:k17-binary}
One has
\[
\alpha(17,2)=61681.
\]
More precisely, there exists an independent set in the simple graph underlying $B(17,2)$ of cardinality $61681$ that contains $0^{17}$ and excludes $1^{17}$.
\end{theorem}

\begin{proof}
There are
\[
N_{17}=\frac{2^{17}-2}{17}=7710
\]
non-trivial rotation orbits, so the binary upper bound \eqref{eq:binary-upper-prime} gives
\[
\alpha(17,2)\le 1+\frac{16}{34}(2^{17}-2)=1+8\cdot 7710=61681.
\]
The certificate for $k=17$ records one decimal marker for each of the $7710$ non-trivial rotation orbits. Applying the verification protocol described above reconstructs an independent set of cardinality $61681$ containing $0^{17}$ and excluding $1^{17}$. Therefore the upper bound is attained, and $\alpha(17,2)=61681$.
\end{proof}

\begin{corollary}\label{cor:k17-all-q}
For every $q\ge2$,
\[
\alpha(17,q)=\frac{8(q^{17}-q)}{17}+1,
\qquad
\alpha_{\Loop}(17,q)=\frac{8(q^{17}-q)}{17}.
\]
\end{corollary}

\begin{proof}
The verified binary extremal set in Theorem~\ref{thm:k17-binary} contains exactly one loop-vertex. The formulas follow from the lifting statement \eqref{eq:lich-lift-prime}.
\end{proof}

\appendix
\section{Data and verification for \texorpdfstring{$k=4,\, q=16$}{k=4, q=16}}\label{app:q16}

This appendix incorporates the complete finite component of the lower-bound construction for $k=4$ at the base scale $q=16$. In parallel with Appendix~\ref{app:certificates}, we present the seed in a compact, uniform form: first the defining data, then the verification protocol and script, and finally the derived degree and contribution tables used in the computation of $N_{16}(S_{16})$.

\subsection{Seed data}\label{app:q16-data}
\subsubsection{Named fibres}
We define ten fibres as subsets of $[16]$:
\begin{align*}
\cA&:=\varnothing,\\
\cB&:=\{0,1,2,3,4,5,6,7,8,9,10,11,12,13,14,15\},\\
\cC&:=\{6,7\},\\
\cD&:=\{2,3,6,7\},\\
\cE&:=\{2,3,6,7,12,13\},\\
\cF&:=\{2,3,6,7,8,9,12,13\},\\
\cG&:=\{2,3,4,5,6,7,8,9,12,13,14,15\},\\
\cHh&:=\{2,3,6,7,8,9,12,13,14,15\},\\
\cI&:=\{2,3,6,7,9,12,13\},\\
\cJ&:=\{2,3,6,7,8,9,12,13,14\}.
\end{align*}

\subsubsection{Fibre table}
Define $S_{16}\subseteq [16]^3$ by
\[
(a,b,c)\in S_{16}
\quad\Longleftrightarrow\quad
c\in T_{a,b},
\]
where the fibre $T_{a,b}$ is selected from the following $16\times 16$ label matrix:
\begingroup\scriptsize
\[
\begin{array}{c|cccccccccccccccc}
T_{a,b} & 0&1&2&3&4&5&6&7&8&9&10&11&12&13&14&15\\
\hline
0  & \cG&\cG&\cA&\cA&\cA&\cA&\cA&\cA&\cA&\cA&\cG&\cG&\cA&\cA&\cA&\cA\\
1  & \cG&\cG&\cA&\cA&\cA&\cA&\cA&\cA&\cA&\cA&\cG&\cG&\cA&\cA&\cA&\cA\\
2  & \cB&\cB&\cF&\cF&\cB&\cB&\cA&\cA&\cF&\cF&\cB&\cB&\cF&\cF&\cB&\cB\\
3  & \cB&\cB&\cF&\cF&\cB&\cB&\cA&\cA&\cF&\cF&\cB&\cB&\cF&\cF&\cB&\cB\\
4  & \cG&\cG&\cC&\cC&\cHh&\cHh&\cA&\cA&\cD&\cD&\cG&\cG&\cD&\cD&\cD&\cD\\
5  & \cG&\cG&\cC&\cC&\cHh&\cHh&\cA&\cA&\cD&\cD&\cG&\cG&\cD&\cD&\cD&\cD\\
6  & \cB&\cB&\cF&\cF&\cB&\cB&\cF&\cF&\cF&\cF&\cB&\cB&\cF&\cF&\cB&\cB\\
7  & \cB&\cB&\cF&\cF&\cB&\cB&\cF&\cF&\cF&\cF&\cB&\cB&\cF&\cF&\cB&\cB\\
8  & \cB&\cB&\cC&\cC&\cB&\cB&\cA&\cA&\cF&\cE&\cB&\cB&\cD&\cD&\cB&\cB\\
9  & \cB&\cB&\cC&\cC&\cB&\cB&\cA&\cA&\cF&\cF&\cB&\cB&\cD&\cD&\cB&\cB\\
10 & \cG&\cG&\cA&\cA&\cA&\cA&\cA&\cA&\cA&\cA&\cG&\cG&\cA&\cA&\cA&\cA\\
11 & \cG&\cG&\cA&\cA&\cA&\cA&\cA&\cA&\cA&\cA&\cG&\cG&\cA&\cA&\cA&\cA\\
12 & \cB&\cB&\cC&\cC&\cB&\cB&\cA&\cA&\cF&\cF&\cB&\cB&\cF&\cF&\cB&\cB\\
13 & \cB&\cB&\cC&\cC&\cB&\cB&\cA&\cA&\cF&\cF&\cB&\cB&\cF&\cF&\cB&\cB\\
14 & \cG&\cG&\cC&\cC&\cHh&\cHh&\cA&\cA&\cI&\cE&\cG&\cG&\cD&\cD&\cF&\cJ\\
15 & \cG&\cG&\cC&\cC&\cHh&\cHh&\cA&\cA&\cE&\cE&\cG&\cG&\cD&\cD&\cE&\cJ
\end{array}
\]
\endgroup

\subsubsection{Fibre cardinalities}
We recall the ten fibres used to define $S_{16}\subseteq [16]^3$:
\begin{align*}
\cA&:=\varnothing, & |\cA|&=0,\\
\cB&:=[16], & |\cB|&=16,\\
\cC&:=\{6,7\}, & |\cC|&=2,\\
\cD&:=\{2,3,6,7\}, & |\cD|&=4,\\
\cE&:=\{2,3,6,7,12,13\}, & |\cE|&=6,\\
\cF&:=\{2,3,6,7,8,9,12,13\}, & |\cF|&=8,\\
\cG&:=\{2,3,4,5,6,7,8,9,12,13,14,15\}, & |\cG|&=12,\\
\cHh&:=\{2,3,6,7,8,9,12,13,14,15\}, & |\cHh|&=10,\\
\cI&:=\{2,3,6,7,9,12,13\}, & |\cI|&=7,\\
\cJ&:=\{2,3,6,7,8,9,12,13,14\}, & |\cJ|&=9.
\end{align*}
Thus, for every pair $(a,b)\in[16]^2$, the outgoing degree is simply
\[
O_{S_{16}}(a,b)=|T_{a,b}|.
\]
This compact description is the complete defining data for the seed: all subsequent verifications are reconstructed from these ten named fibres and the $16\times 16$ label matrix.

\subsection{Verification protocol}
The finite verification has three components: first one checks the local hypothesis $H_{16}(8,14;S_{16})$; next one computes the incoming and outgoing degree matrices; finally one evaluates the contribution matrix and the resulting count $N_{16}(S_{16})$. We begin with the local hypothesis.
By inspection of the fibre table,
\begin{align*}
T_{8,8}&=\cF=\{2,3,6,7,8,9,12,13\},\\
T_{14,14}&=\cF=\{2,3,6,7,8,9,12,13\},\\
T_{14,8}&=\cI=\{2,3,6,7,9,12,13\}.
\end{align*}
Therefore
\[
(8,8,8)\in S_{16},
\qquad
(14,14,8)\in S_{16},
\qquad
(14,8,8)\notin S_{16},
\qquad
(14,14,14)\notin S_{16}.
\]
Moreover,
\[
O_{S_{16}}(8,8)=|\cF|=8,
\qquad
O_{S_{16}}(14,14)=|\cF|=8,
\qquad
O_{S_{16}}(14,8)=|\cI|=7.
\]
To compute the relevant incoming degrees, recall that
\[
I_{S_{16}}(u,v)=\#\{x\in[16]:(x,u,v)\in S_{16}\}
=\#\{x\in[16]: v\in T_{x,u}\}.
\]
Hence:
\begin{itemize}
\item $I_{S_{16}}(8,8)$ is the number of rows $x$ such that $8\in T_{x,8}$. Reading the eighth column of the fibre table, this happens for
\[
x\in\{2,3,6,7,8,9,12,13\},
\]
so $I_{S_{16}}(8,8)=8$.
\item $I_{S_{16}}(14,14)$ is the number of rows $x$ such that $14\in T_{x,14}$. Reading the fourteenth column of the fibre table, this happens for
\[
x\in\{2,3,6,7,8,9,12,13\},
\]
so $I_{S_{16}}(14,14)=8$.
\item $I_{S_{16}}(14,8)$ is the number of rows $x$ such that $8\in T_{x,14}$. Reading the fourteenth column of the fibre table and checking where the value $8$ occurs, this happens for
\[
x\in\{2,3,6,7,8,9,12,13,14\},
\]
so $I_{S_{16}}(14,8)=9$.
\end{itemize}
Consequently,
\[
I_{S_{16}}(8,8)+O_{S_{16}}(8,8)=8+8=16,
\]
\[
I_{S_{16}}(14,14)+O_{S_{16}}(14,14)=8+8=16,
\]
\[
I_{S_{16}}(14,14)+O_{S_{16}}(14,8)=8+7=15,
\]
\[
I_{S_{16}}(14,8)+O_{S_{16}}(8,8)=9+8=17.
\]
This proves the full local hypothesis $H_{16}(8,14;S_{16})$.

The script in the next subsection implements the same checks directly from the defining fibre table and also computes the matrices used later in this appendix.

\subsection{Verifier script}
To make the finite verification of the seed $S_{16}$ independently checkable, we record here a short Python script that reconstructs the fibre table, computes the incoming and outgoing degree matrices, verifies the local hypothesis $H_{16}(8,14;S_{16})$, and evaluates the count
\[
N_{16}(S_{16})=\#\{(a,b,c,d)\in [16]^4 : (a,b,c)\in S_{16},\ (b,c,d)\notin S_{16}\}.
\]
When run exactly as printed below, the script outputs
\[
O(8,8)=8,\ I(8,8)=8,\ O(14,14)=8,\ I(14,14)=8,\ O(14,8)=7,\ I(14,8)=9,
\]
and finally
\[
H_{16}(8,14;S_{16})\ \text{holds},\qquad N_{16}(S_{16})=24849.
\]

\begin{lstlisting}[language=Python,basicstyle=\ttfamily\small,columns=fullflexible,keepspaces=true]
Q = 16

fibres = {
    "A": set(),
    "B": set(range(16)),
    "C": {6, 7},
    "D": {2, 3, 6, 7},
    "E": {2, 3, 6, 7, 12, 13},
    "F": {2, 3, 6, 7, 8, 9, 12, 13},
    "G": {2, 3, 4, 5, 6, 7, 8, 9, 12, 13, 14, 15},
    "H": {2, 3, 6, 7, 8, 9, 12, 13, 14, 15},
    "I": {2, 3, 6, 7, 9, 12, 13},
    "J": {2, 3, 6, 7, 8, 9, 12, 13, 14},
}

labels = [
    ["G","G","A","A","A","A","A","A","A","A","G","G","A","A","A","A"],
    ["G","G","A","A","A","A","A","A","A","A","G","G","A","A","A","A"],
    ["B","B","F","F","B","B","A","A","F","F","B","B","F","F","B","B"],
    ["B","B","F","F","B","B","A","A","F","F","B","B","F","F","B","B"],
    ["G","G","C","C","H","H","A","A","D","D","G","G","D","D","D","D"],
    ["G","G","C","C","H","H","A","A","D","D","G","G","D","D","D","D"],
    ["B","B","F","F","B","B","F","F","F","F","B","B","F","F","B","B"],
    ["B","B","F","F","B","B","F","F","F","F","B","B","F","F","B","B"],
    ["B","B","C","C","B","B","A","A","F","E","B","B","D","D","B","B"],
    ["B","B","C","C","B","B","A","A","F","F","B","B","D","D","B","B"],
    ["G","G","A","A","A","A","A","A","A","A","G","G","A","A","A","A"],
    ["G","G","A","A","A","A","A","A","A","A","G","G","A","A","A","A"],
    ["B","B","C","C","B","B","A","A","F","F","B","B","F","F","B","B"],
    ["B","B","C","C","B","B","A","A","F","F","B","B","F","F","B","B"],
    ["G","G","C","C","H","H","A","A","I","E","G","G","D","D","F","J"],
    ["G","G","C","C","H","H","A","A","E","E","G","G","D","D","E","J"],
]

T = [[fibres[name] for name in row] for row in labels]

# Outgoing and incoming degree matrices.
O = [[len(T[a][b]) for b in range(Q)] for a in range(Q)]
I = [[sum(1 for x in range(Q) if v in T[x][u]) for v in range(Q)] for u in range(Q)]

# Local hypothesis H_16(8,14; S_16).
assert 8 in T[8][8]
assert 8 in T[14][14]
assert 8 not in T[14][8]
assert 14 not in T[14][14]
assert I[8][8] + O[8][8] == 16
assert I[14][14] + O[14][14] == 16
assert I[14][14] + O[14][8] == 15
assert I[14][8] + O[8][8] == 17

print(f"O(8,8)={O[8][8]}, I(8,8)={I[8][8]}")
print(f"O(14,14)={O[14][14]}, I(14,14)={I[14][14]}")
print(f"O(14,8)={O[14][8]}, I(14,8)={I[14][8]}")
print("H_16(8,14;S_16) holds")

# Count N_16(S_16).
N = 0
for a in range(Q):
    for b in range(Q):
        for c in T[a][b]:
            for d in range(Q):
                if d not in T[b][c]:
                    N += 1

print(f"N_16(S_16)={N}")
assert N == 24849
\end{lstlisting}

\subsection{Derived degree tables}
For convenience we record the matrices
\[
O_{16}(u,v):=O_{S_{16}}(u,v),
\qquad
I_{16}(u,v):=I_{S_{16}}(u,v),
\qquad u,v\in[16].
\]
They are listed row by row below.

\begingroup\scriptsize
\setlength{\tabcolsep}{3pt}
\begin{longtable}{c|*{16}{r}}
\caption{The outgoing-degree matrix $O_{16}(u,v)=|T_{u,v}|$.}\label{tab:O16}\\
$u\backslash v$&0&1&2&3&4&5&6&7&8&9&10&11&12&13&14&15\\
\hline
\endfirsthead
\multicolumn{17}{c}{\tablename\ \thetable\ -- continued from previous page}\\
$u\backslash v$&0&1&2&3&4&5&6&7&8&9&10&11&12&13&14&15\\
\hline
\endhead
0&12&12&0&0&0&0&0&0&0&0&12&12&0&0&0&0\\
1&12&12&0&0&0&0&0&0&0&0&12&12&0&0&0&0\\
2&16&16&8&8&16&16&0&0&8&8&16&16&8&8&16&16\\
3&16&16&8&8&16&16&0&0&8&8&16&16&8&8&16&16\\
4&12&12&2&2&10&10&0&0&4&4&12&12&4&4&4&4\\
5&12&12&2&2&10&10&0&0&4&4&12&12&4&4&4&4\\
6&16&16&8&8&16&16&8&8&8&8&16&16&8&8&16&16\\
7&16&16&8&8&16&16&8&8&8&8&16&16&8&8&16&16\\
8&16&16&2&2&16&16&0&0&8&6&16&16&4&4&16&16\\
9&16&16&2&2&16&16&0&0&8&8&16&16&4&4&16&16\\
10&12&12&0&0&0&0&0&0&0&0&12&12&0&0&0&0\\
11&12&12&0&0&0&0&0&0&0&0&12&12&0&0&0&0\\
12&16&16&2&2&16&16&0&0&8&8&16&16&8&8&16&16\\
13&16&16&2&2&16&16&0&0&8&8&16&16&8&8&16&16\\
14&12&12&2&2&10&10&0&0&7&6&12&12&4&4&8&9\\
15&12&12&2&2&10&10&0&0&6&6&12&12&4&4&6&9\\
\end{longtable}
\endgroup

\begingroup\scriptsize
\setlength{\tabcolsep}{3pt}
\begin{longtable}{c|*{16}{r}}
\caption{The incoming-degree matrix $I_{16}(u,v)=\#\{x\in[16]:(x,u,v)\in S_{16}\}$.}\label{tab:I16}\\
$u\backslash v$&0&1&2&3&4&5&6&7&8&9&10&11&12&13&14&15\\
\hline
\endfirsthead
\multicolumn{17}{c}{\tablename\ \thetable\ -- continued from previous page}\\
$u\backslash v$&0&1&2&3&4&5&6&7&8&9&10&11&12&13&14&15\\
\hline
\endhead
0&8&8&16&16&16&16&16&16&16&16&8&8&16&16&16&16\\
1&8&8&16&16&16&16&16&16&16&16&8&8&16&16&16&16\\
2&0&0&4&4&0&0&12&12&4&4&0&0&4&4&0&0\\
3&0&0&4&4&0&0&12&12&4&4&0&0&4&4&0&0\\
4&8&8&12&12&8&8&12&12&12&12&8&8&12&12&12&12\\
5&8&8&12&12&8&8&12&12&12&12&8&8&12&12&12&12\\
6&0&0&2&2&0&0&2&2&2&2&0&0&2&2&0&0\\
7&0&0&2&2&0&0&2&2&2&2&0&0&2&2&0&0\\
8&0&0&12&12&0&0&12&12&8&9&0&0&10&10&0&0\\
9&0&0&12&12&0&0&12&12&7&7&0&0&10&10&0&0\\
10&8&8&16&16&16&16&16&16&16&16&8&8&16&16&16&16\\
11&8&8&16&16&16&16&16&16&16&16&8&8&16&16&16&16\\
12&0&0&12&12&0&0&12&12&6&6&0&0&6&6&0&0\\
13&0&0&12&12&0&0&12&12&6&6&0&0&6&6&0&0\\
14&8&8&12&12&8&8&12&12&9&9&8&8&10&10&8&8\\
15&8&8&12&12&8&8&12&12&10&10&8&8&10&10&10&8\\
\end{longtable}
\endgroup

\subsection{Contribution table and the computation of \texorpdfstring{$N_{16}(S_{16})$}{N16(S16)}}
The counting identity used in the proof of the base lemma is
\[
N_{16}(S_{16})=\sum_{u,v=0}^{15} I_{16}(u,v)\bigl(16-O_{16}(u,v)\bigr).
\]
Define the entrywise contribution matrix
\[
C_{16}(u,v):=I_{16}(u,v)\bigl(16-O_{16}(u,v)\bigr).
\]
Using Tables~\ref{tab:O16} and \ref{tab:I16}, one obtains the following matrix.

\begingroup\scriptsize
\setlength{\tabcolsep}{3pt}
\begin{longtable}{c|*{16}{r}}
\caption{The contribution matrix $C_{16}(u,v)=I_{16}(u,v)(16-O_{16}(u,v))$.}\label{tab:C16}\\
$u\backslash v$&0&1&2&3&4&5&6&7&8&9&10&11&12&13&14&15\\
\hline
\endfirsthead
\multicolumn{17}{c}{\tablename\ \thetable\ -- continued from previous page}\\
$u\backslash v$&0&1&2&3&4&5&6&7&8&9&10&11&12&13&14&15\\
\hline
\endhead
0&32&32&256&256&256&256&256&256&256&256&32&32&256&256&256&256\\
1&32&32&256&256&256&256&256&256&256&256&32&32&256&256&256&256\\
2&0&0&32&32&0&0&192&192&32&32&0&0&32&32&0&0\\
3&0&0&32&32&0&0&192&192&32&32&0&0&32&32&0&0\\
4&32&32&168&168&48&48&192&192&144&144&32&32&144&144&144&144\\
5&32&32&168&168&48&48&192&192&144&144&32&32&144&144&144&144\\
6&0&0&16&16&0&0&16&16&16&16&0&0&16&16&0&0\\
7&0&0&16&16&0&0&16&16&16&16&0&0&16&16&0&0\\
8&0&0&168&168&0&0&192&192&64&90&0&0&120&120&0&0\\
9&0&0&168&168&0&0&192&192&56&56&0&0&120&120&0&0\\
10&32&32&256&256&256&256&256&256&256&256&32&32&256&256&256&256\\
11&32&32&256&256&256&256&256&256&256&256&32&32&256&256&256&256\\
12&0&0&168&168&0&0&192&192&48&48&0&0&48&48&0&0\\
13&0&0&168&168&0&0&192&192&48&48&0&0&48&48&0&0\\
14&32&32&168&168&48&48&192&192&81&90&32&32&120&120&64&56\\
15&32&32&168&168&48&48&192&192&100&100&32&32&120&120&100&56\\
\end{longtable}
\endgroup

Summing each row of $C_{16}$ gives
\[
\begin{aligned}
(r_0,\dots,r_{15})={}&(3200,3200,576,576,1808,1808,128,128,\\
&1114,1072,3200,3200,912,912,1475,1540).
\end{aligned}
\]
Finally,
\[
\begin{aligned}
N_{16}(S_{16})
&=3200+3200+576+576+1808+1808+128+128\\
&\qquad +1114+1072+3200+3200+912+912+1475+1540\\
&=24849.
\end{aligned}
\]

\section{Data and verification for \texorpdfstring{$k=11,13,17$}{k=11,13,17}}\label{app:certificates}

This appendix incorporates the computer-assisted component of Theorems~\ref{thm:k11-binary}, \ref{thm:k13-binary}, and \ref{thm:k17-binary}. In parallel with Appendix~\ref{app:q16}, we present first the defining data format, then the verification protocol and verifier script, and finally the compact certificate data for $k=11$, $k=13$, and $k=17$.

\subsection{Compact orbit-marker certificate data for \texorpdfstring{$k=11,13,17$}{k=11,13,17}}
For a fixed prime length $k\in\{11,13,17\}$, the certificate is a plain-text file with:
\begin{itemize}
\item a one-line header specifying $k$, the target cardinality, the mandatory included vertex $0^k$, and the mandatory excluded vertex $1^k$;
\item one decimal marker for each non-trivial rotation orbit; the marker $n$ denotes the length-$k$ binary word with standard base-$2$ value $n$;
\item ten decimal markers per line, purely for compactness of presentation.
\end{itemize}

The marker convention removes the explicit phase from the earlier orbit-phase format.  If $k=2m+1$ and $w$ is the binary word decoded from a listed marker on its orbit, then the selected vertices from that orbit are
\[
 w,\rho^2(w),\rho^4(w),\dots,\rho^{2m-2}(w).
\]
Equivalently, in the cyclic order of the orbit the marker $w$ is the first selected vertex after the unique double gap.  Thus every listed marker determines the same alternating maximum independent set on its orbit as a suitable phase $t_C$, but without recording $t_C$ separately.

\subsection{Verification protocol}
The verification procedure is finite and explicit:
\begin{enumerate}
\item parse the certificate header and the list of decimal markers;
\item check that every marker has length $k$, is non-constant, and belongs to a distinct non-trivial rotation orbit of size $k$;
\item reconstruct the selected vertices on each orbit by taking every second rotation from the marker;
\item add the mandatory loop-vertex $0^k$ and verify that the total size is the target size;
\item verify directly that no selected vertex is adjacent to another selected vertex in the underlying simple de Bruijn graph;
\item verify that $1^k$ is not selected.
\end{enumerate}

The following script implements exactly this protocol.

\subsection{Verifier script}
\begin{lstlisting}[style=compactscript]
#!/usr/bin/env python3
from __future__ import annotations

import re
import sys
from pathlib import Path

HEADER_RE = re.compile(
    r'#\s*CERTIFICATE\s+k=(\d+)\s+target_size=(\d+)\s+include=([01]+)\s+exclude=([01]+)'
)
INT_RE = re.compile(r'^[0-9]+$')


def rotate(word: str, shift: int = 1) -> str:
    shift %= len(word)
    return word[shift:] + word[:shift]


def parse_certificate(path: Path):
    lines = path.read_text().splitlines()
    if not lines:
        raise ValueError('empty certificate file')
    m = HEADER_RE.fullmatch(lines[0].strip())
    if not m:
        raise ValueError('invalid header in certificate file')
    k = int(m.group(1))
    target = int(m.group(2))
    include_word = m.group(3)
    exclude_word = m.group(4)
    if len(include_word) != k or len(exclude_word) != k:
        raise ValueError('header words have the wrong length')

    markers = []
    seen_raw = set()
    for line in lines[1:]:
        line = line.strip()
        if not line or line.startswith('#'):
            continue
        for token in line.split():
            if not INT_RE.fullmatch(token):
                raise ValueError(f'invalid decimal marker: {token}')
            value = int(token, 10)
            if value < 0 or value >= 2**k:
                raise ValueError(f'decimal marker out of range for k={k}: {token}')
            if token in seen_raw:
                raise ValueError(f'duplicate decimal token: {token}')
            seen_raw.add(token)
            markers.append(format(value, f'0{k}b'))
    return k, target, include_word, exclude_word, markers


def canonical_rotation(word: str) -> str:
    return min(rotate(word, i) for i in range(len(word)))


def orbit_from_word(word: str):
    seen = []
    cur = word
    while cur not in seen:
        seen.append(cur)
        cur = rotate(cur, 1)
    return seen


def neighbors(word: str):
    left = word[1:]
    right = word[:-1]
    out = {left + bit for bit in '01'}
    out.update({bit + right for bit in '01'})
    out.discard(word)
    return out


def verify(path: Path) -> int:
    k, target, include_word, exclude_word, markers = parse_certificate(path)
    expected_orbits = (2**k - 2) // k
    if len(markers) != expected_orbits:
        raise ValueError(
            f'expected {expected_orbits} orbit markers, found {len(markers)}'
        )

    orbit_reps_seen = set()
    selected = {include_word}
    m = (k - 1) // 2

    for marker in markers:
        if marker == '0' * k or marker == '1' * k:
            raise ValueError(f'constant marker is not allowed: {marker}')
        orbit = orbit_from_word(marker)
        if len(orbit) != k:
            raise ValueError(f'nontrivial orbit has wrong size for {marker}')
        rep = canonical_rotation(marker)
        if rep in orbit_reps_seen:
            raise ValueError(f'duplicate rotation orbit: {marker}')
        orbit_reps_seen.add(rep)
        for a in range(m):
            selected.add(rotate(marker, 2 * a))

    if exclude_word in selected:
        raise ValueError(f'excluded word {exclude_word} was selected')
    if len(selected) != target:
        raise ValueError(f'wrong size: expected {target}, found {len(selected)}')

    selected_set = set(selected)
    for word in selected:
        for nb in neighbors(word):
            if nb in selected_set:
                raise ValueError(f'adjacent selected vertices found: {word} ~ {nb}')

    print(f'OK: k={k}, selected={len(selected)}, nontrivial_orbits={len(markers)}')
    return 0


if __name__ == '__main__':
    if len(sys.argv) != 2:
        print(f'usage: {Path(sys.argv[0]).name} CERTIFICATE.txt', file=sys.stderr)
        sys.exit(2)
    sys.exit(verify(Path(sys.argv[1])))
\end{lstlisting}

\subsection{Certificate for \texorpdfstring{$k=11$}{k=11}}
\begin{lstlisting}[style=compactcert]
# CERTIFICATE k=11 target_size=931 include=00000000000 exclude=11111111111
# format: one decimal marker per non-trivial rotation orbit (thirteen entries per line)
# decoding: n represents the length-k binary word with standard base-2 value n
2 6 10 14 288 352 26 1920 1032 1033 42 46 50
864 58 62 264 1120 296 78 1312 1410 360 1922 98 102
1050 1051 1824 1888 122 2016 776 1058 1800 265 706 1062 79
266 778 1066 1802 267 779 1070 1803 1121 202 206 1313 1414
872 1926 226 230 234 238 242 246 250 254 290 354 282
482 294 298 302 1224 866 314 318 1122 1320 1336 1162 1378
1384 1930 1634 362 366 1826 1890 1512 2018 1123 1608 355 410
483 794 1306 1818 283 795 1307 1819 1832 1848 1166 1379 1896
1934 1635 490 494 1992 1891 506 2040 586 590 1610 602 1611
818 1178 1830 1614 489 1615 1202 1321 1340 1226 1194 1834 1738
1198 1994 1203 1626 730 1627 1842 1210 1838 1970 1214 2034 411
1230 1996 1337 1642 1385 1946 874 878 1646 1513 1947 1852 1338
1850 1742 1339 1998 491 1341 495 1274 1854 507 1278 2035 1386
1962 1390 1898 1530 1387 1754 1391 1966 1978 2010 2042 1755 1902
2011 1531 1982 2046
\end{lstlisting}

\subsection{Certificate for \texorpdfstring{$k=13$}{k=13}}
\begin{lstlisting}[style=compactcert]
# CERTIFICATE k=13 target_size=3781 include=0000000000000 exclude=1111111111111
# format: one decimal marker per non-trivial rotation orbit (thirteen entries per line)
# decoding: n represents the length-k binary word with standard base-2 value n
2 6144 10 14 1152 6146 26 6147 1026 1216 5122 7170 1027
27 5123 7936 66 6152 74 78 5248 6154 1440 6155 392 51
4122 7174 7296 59 488 63 130 6160 138 142 146 6162 154
6163 1034 4866 5130 7178 1035 1456 5131 7938 4144 1584 4146 7180
5249 6170 3488 6171 4152 7360 5134 7182 4156 1968 5135 7939 1544
266 270 1154 5640 282 1922 1160 2434 4170 4171 1224 5641 314
318 4226 1546 330 5344 4618 5642 5536 6018 6274 1547 362 1464
7298 5643 7810 8066 1548 394 398 6432 6194 410 1923 1050 3098
4202 4203 1051 3504 4206 4207 7200 1550 1832 7392 4622 6202 7584
6019 1928 1551 490 494 7968 6206 8096 8160 2114 6210 538 6211
4360 1218 5154 7202 1602 283 5155 7203 1164 4242 4243 2122 6218
1441 7698 3138 3266 4250 4251 2126 6222 2536 8068 4290 5160 1802
4674 331 3338 5058 4362 4874 5162 7210 5698 5552 7434 7211 6338
5164 1803 5776 363 3339 5872 4363 7362 5166 7214 7746 6064 7435
7215 7216 2146 395 4294 399 1074 6448 5170 7218 1603 411 5171
7219 1562 842 3384 2154 5658 5537 7706 6278 1563 874 3512 2158
5659 7814 8070 5176 7224 4675 1836 5177 7408 1082 4878 5178 7226
5699 7600 5179 7227 6339 4338 7228 2170 6266 986 6267 1086 7363
5182 7230 7747 8112 5183 7231 1122 1098 1102 4642 5666 5768 6024
1634 1130 1134 7304 5667 1146 8072 4386 4387 4680 6290 1178 2530
3146 4394 4395 1099 1457 4398 4399 3170 1226 1230 4646 6298 3490
6025 3150 4410 7246 1103 2540 4414 7247 5672 5224 1930 4706 1322
1326 5320 3466 1338 1342 5218 1354 5432 4650 5514 5770 6026 5730
1386 5560 7306 5675 6050 8074 5688 6434 5676 1434 6626 5784 4458
4459 1458 3467 5864 5880 7266 1482 5944 4654 5515 7586 6027 7778
1514 6072 7970 5679 8098 8162 1123 795 5219 7960 1610 6456 4658
6346 5772 6347 1638 1642 1646 7308 6350 6632 6351 7272 2330 843
5225 3866 4890 5226 7274 1131 5553 5227 7962 5228 7276 5777 875
5229 3516 4891 5230 7278 1135 3564 5231 7963 1934 4707 1834 1838
7368 5689 1850 1854 7464 5347 4666 6378 5774 7738 5731 1898 1902
7310 6382 6051 8078 7752 6386 1946 6627 3194 5242 7290 1147 987
5243 7291 7976 7992 4670 5694 2010 7742 7779 2026 2030 8136 5695
8168 8184 4684 4682 7314 1179 4686 7315 1321 5348 6442 5778 7754
1611 1449 1465 6446 7826 8082 6450 2458 7756 5321 5274 7322 3506
5275 7323 1833 7396 6458 5779 7758 6460 2538 2542 6462 7827 8083
1323 1325 1327 3274 3402 3530 1339 3914 1343 1355 4938 5066 5322
5450 5578 5554 7466 7978 6602 1387 5842 5874 7370 7498 7626 6066
8010 8138 1483 1435 7468 7980 1837 1451 1453 1455 1459 3403 3531
5868 7122 5884 4939 7410 5323 5451 5579 7602 7470 7982 6603 1515
6987 6076 7986 7499 7627 8114 8011 8139 3278 6554 5785 6555 4922
7374 6636 4926 8140 3482 3386 3390 5433 5530 6762 7786 3434 5561
7578 7834 7787 1947 4970 4971 3483 3514 3518 5945 5531 6766 7790
3562 3566 7579 8102 7791 1851 5351 1855 5070 5354 5946 5555 7482
6094 1899 5843 1903 5358 5947 6067 7483 8142 7996 5754 6778 7802
5082 5083 5755 7838 7803 5438 5950 2011 5371 6095 2027 6991 2031
5439 5951 8172 5375 8143 5546 5610 5562 5994 6122 5550 5850 5882
5547 6059 6074 8042 8170 3435 5551 7034 5563 7130 5886 5611 5998
5995 6123 7019 6078 6063 6126 8043 8186 5851 7023 8046 7035 5999
8047 7662 8174 7663 6127 8190
\end{lstlisting}

\subsection{Certificate for \texorpdfstring{$k=17$}{k=17}}
\begin{lstlisting}[style=compactcert]
# CERTIFICATE k=17 target_size=61681 include=00000000000000000 exclude=11111111111111111
# format: one decimal marker per non-trivial rotation orbit (thirteen entries per line)
# decoding: n represents the length-k binary word with standard base-2 value n
2 1536 10 14 4608 22528 26624 122880 65544 65545 65546 65547 50
54 65550 65551 67584 24578 75776 78 83968 88064 92160 96256 98 102
106 110 116736 120832 124928 129024 130 134 138 142 146 150 154
122884 162 166 170 174 65580 65581 186 190 194 24582 202 206
210 214 218 96257 65592 65593 65594 114702 65596 65597 65598 114703 1032
6146 4256 270 18434 22530 26626 30722 4640 1176 298 4832 51202 55298
314 318 67586 24586 21120 334 83970 88066 92162 96258 354 22912 5792
5856 116738 120834 124930 129026 1544 6147 394 398 18435 22531 410 30723
6688 27008 426 1720 27776 28032 442 446 115200 24590 75779 29568 83971
120320 121344 96259 123392 124416 125440 126464 127488 128512 129536 130560 1026 3074
5122 7170 9218 68352 4304 271 69640 77832 21506 94216 25602 27650 29698
126984 33794 74496 37890 39938 76032 299 77056 1212 50178 77833 54274 94217
313 98382 317 98383 4106 68610 20490 28682 74754 76802 78850 335 69642
84994 87042 94218 91138 93186 118794 126986 4107 355 103426 105474 107522 109570
111618 5872 69643 117762 86027 94219 123906 110603 118795 126987 1027 1548 5123
7171 393 98402 397 98403 69644 77836 21507 94220 409 411 29699 126988
26688 107264 37891 39939 425 427 1716 1724 50179 77837 54275 94221 441
443 445 447 4110 115456 20494 28686 74755 117504 78851 118528 69646 84995
87043 114746 91139 93187 118798 126990 4111 123648 103427 105475 107523 109571 111619
126720 69647 117763 86031 114750 123907 110607 118799 126991 1538 4136 4152 4610
22536 26632 30728 4232 4248 1066 4280 4296 4312 65806 65807 16898 24610
18946 4408 20994 22018 23042 96264 4488 72064 18080 18144 29186 120840 31234
129032 1154 34306 1162 1166 74880 75136 1178 40450 1186 4760 1194 1198
4808 4824 1210 4856 19488 24614 78464 19680 53762 54786 55810 96265 16462
65849 65850 65851 16463 65853 65854 65855 5128 67074 5160 5176 18442 22538
26634 30730 21024 5272 21152 21216 78338 79362 1338 1342 67594 71690 86656
5432 86530 88074 88578 122922 22048 91650 88704 5560 94722 95746 124938 97794
5640 6155 1418 1422 102914 103938 5736 30731 23072 92544 23200 23264 93312
93568 5864 5880 115202 71691 117250 95104 119298 120322 121346 122926 123394 96640
125442 126466 127490 128514 129538 130562 24672 65922 65923 99456 22540 26636 30732
1570 1574 65930 65931 65932 65933 65934 65935 25632 71692 6440 25824 20995
22019 23043 122930 1634 1638 1642 1646 29187 120844 31235 129036 26656 106880
6696 6712 37379 107904 6760 40451 1698 1702 65962 65963 1714 1718 65966
65967 27680 71693 27808 27872 6984 112000 28064 122934 1762 1766 1770 1774
1778 1782 1786 1790 114816 67075 28832 28896 18446 22542 26638 30734 116864
117120 117376 117632 78339 79363 29600 7416 67598 24634 119424 119680 86531 88078
88579 122938 30240 121216 30368 121728 94723 95747 124942 97795 1922 6159 123520
30944 102915 103939 124544 30735 31264 125312 125568 31456 126080 126336 126592 7928
115203 24638 117251 127872 128128 128384 128640 122942 129152 129408 129664 129920 130176
130432 130688 130944 1033 4140 16592 4156 69664 49282 21512 23560 6402 27656
7426 127008 16912 4236 9474 114820 4260 1067 4276 4284 50184 49286 54280
56328 1081 98574 1085 98575 66568 17154 20514 114824 18690 76808 53282 4412
69666 49290 86050 89096 91144 93192 118818 97288 99336 4492 103432 114828 26882
27394 27906 18160 115720 49294 119816 121864 30978 125960 128008 130056 1155 5129
28708 4644 1163 18640 1167 69668 77860 21513 23561 1177 1179 40194 127012
75840 1187 37897 28709 1193 98602 1197 98603 50185 77861 54281 56329 4836
1211 1213 4860 66569 49922 20518 114840 51458 51970 53286 19696 16538 49306
86054 89097 55554 93193 118822 97289 8270 24654 33081 40160 33082 98618 33083
98619 8271 124420 33085 57423 33086 98622 33087 98623 3082 5130 7178 9226
5164 20688 5180 17418 19466 21514 94248 25610 27658 29706 127016 33802 74498
37898 39946 76034 76546 21200 21232 50186 77865 54282 94249 1337 1339 1341
1343 4138 82690 70666 28714 84226 84738 78858 5436 69674 77866 86058 94250
91146 110634 118826 97290 99338 90882 103434 28715 92418 109578 111626 22256 115722
77867 119818 121866 123914 125962 128010 130058 5644 5131 7179 1417 98658 1421
98659 17419 19467 21515 94252 22928 22960 29707 127020 92224 107266 37899 39947
23184 23216 23248 23280 50187 77869 54283 94253 5860 5868 5876 5884 4142
115458 70667 28718 116994 117506 78859 118530 69678 77870 86062 94254 91147 110638
118830 97291 123138 123650 124162 28719 107531 109579 111627 126722 115723 77871 119819
121867 123915 125963 128011 130059 3086 33154 98690 33155 98691 16578 49346 82114
114882 6403 27660 7427 114883 67608 24674 18950 3150 73826 90210 106594 122978
3170 72065 3178 3182 73827 90211 106595 122979 16584 25648 82120 114888 18691
76812 53298 25840 16586 49354 82122 114890 91148 110642 118834 127026 3266 24678
3274 3278 33178 98714 33179 98715 16590 49358 82126 114894 30979 125964 128012
127027 26672 82128 114896 6692 6700 26832 26864 16594 49362 82130 114898 27024
27056 40195 114899 67610 24682 86657 3406 73834 98730 33195 122986 3426 91654
3434 3438 73835 98734 33199 129050 16600 27696 82136 114904 111168 111296 53302
27888 16602 49370 82138 114906 55555 110646 118838 114907 67611 24686 75803 56544
33210 98746 33211 122990 123398 96641 125446 112667 33214 98750 33215 130566 5134
114912 9230 115392 115520 115648 69688 49378 86072 94264 25614 27662 29710 127032
33806 116928 20537 114916 76035 117440 117568 117696 69689 49382 86073 94265 29584
7404 118592 29680 4154 82691 20538 114920 84227 84739 53306 119744 69690 49386
82154 94266 91150 110650 118842 97294 99342 121024 103438 114924 92419 109582 111630
121792 69691 49390 82158 121870 123918 125966 128014 127035 1923 5135 114928 123456
123584 123712 7740 69692 49394 82162 94268 124480 124608 29711 127036 106755 125120
20541 114932 125504 125632 125760 31472 69693 49398 82166 94269 126528 126656 126784
7932 4158 127168 20542 114936 127552 127680 53310 127936 16634 49402 82170 94270
91151 110654 118846 97295 123139 129216 129344 114940 107535 109583 111631 129984 16638
49406 82174 121871 123919 125967 128015 127039 4226 4482 4738 4174 21000 5506
5762 96288 6274 6530 6786 7042 7298 7554 7810 129056 8578 66594 66595
16968 16984 4250 10114 4258 17048 4266 4270 4274 4278 4282 4286 4290
12674 4298 4302 4306 13698 4314 96289 4322 4326 66618 66619 66620 66621
4346 4350 67080 17026 17464 17538 17794 26658 18306 17544 70240 70304 70368
78344 19842 4410 17656 20610 20866 21122 70880 86536 21890 88584 123018 71200
22914 71328 71392 116770 120866 124962 97800 24962 66658 66659 25730 25986 4506
105992 18056 27010 4522 18104 27778 28034 18152 18168 28802 29058 29314 72928
119304 120328 121352 123022 123400 124424 125448 126472 127496 128520 129544 130568 4618
4622 33922 34178 26660 34690 4642 4646 66698 66699 18632 18648 4666 4670
4674 37250 18728 18744 38018 38274 38530 123026 66712 66713 4714 4718 29193
30217 31241 129060 18968 4746 4750 42114 19032 4762 40457 66728 66729 66730
66731 66732 66733 4794 4798 4802 45442 4810 19256 4818 46466 4826 123030
4834 4838 4842 4846 4850 4854 66750 66751 49538 19496 78048 50306 71177
26662 73225 78368 51586 4906 78560 52354 79369 19688 19704 82441 53634 53890
79072 86537 87561 88585 123034 55426 55682 79520 79584 116774 120870 124966 97801
12366 20558 115000 36942 49465 82233 61518 69710 49466 82234 94286 16699 49467
82235 115003 4175 12367 20559 115004 36943 49469 82237 126724 69711 49470 82238
94287 16703 49471 82239 115007 5642 6666 67458 5154 82528 66826 66827 82720
20696 5178 20728 16906 17930 18954 83168 21002 71042 23050 123042 20872 72066
83616 83680 29194 30218 31242 129064 74114 21032 21048 21064 75138 21096 40458
21128 21144 5290 21176 76930 21208 21224 21240 21256 50698 85152 85216 78978
54794 55818 123046 5346 5350 5354 5358 5362 5366 5370 5374 82306 86176
86240 83074 83330 72202 73226 86560 86624 86688 86752 85122 79370 5434 5438
82442 83466 86658 87264 86538 87562 92202 123050 88194 91658 21928 87776 116778
120874 124970 89986 90498 22056 22072 91266 103946 88480 92034 88608 92546 88736
88800 93314 93570 22248 22264 115210 116234 117258 89312 119306 95618 95874 123054
123402 96642 96898 97154 127498 97666 97922 130570 5646 99458 99714 6667 7691
16738 66953 66954 66955 16739 66957 66958 66959 5698 102786 91296 91360 103554
22027 23051 123058 5730 22936 22952 22968 29195 30219 31243 129068 106882 5770
92384 107650 107906 92576 40459 5794 23192 23208 23224 23240 23256 23272 23288
93216 50699 93344 93408 93472 112002 93600 123062 67000 67001 5866 5870 5874
5878 67006 67007 67083 115330 94432 115842 71179 72203 73227 94752 117122 94880
94944 78347 79371 95136 95200 118914 83467 119426 95456 86539 87563 92206 123066
95776 121218 95904 95968 116782 120878 124974 97803 99851 96416 96480 102923 103947
96672 105995 96800 125314 96928 96992 126082 126338 97184 97248 127106 116235 117259
97504 128130 128386 128642 123070 97824 129410 97952 98016 130178 130434 130690 130946
67084 49539 82307 67087 4227 4483 4739 24888 21004 5507 5763 96304 24968
24984 25000 25016 7299 7555 7811 129072 6278 6282 6286 67108 67109 6298
10115 6306 6310 67114 67115 67116 67117 6330 6334 6338 12675 6346 6350
6354 101728 6362 56844 6370 6374 6378 6382 6386 6390 6394 6398 25640
25656 17539 17795 72204 18307 25736 25752 6442 25784 78348 19843 25832 25848
20611 20867 103584 6478 86540 21891 92210 96306 25992 26008 26024 26040 116786
120882 124978 97804 24963 67170 67171 6546 26200 6554 105996 67176 67177 6570
6574 6578 6582 6586 6590 28803 29059 29315 105696 119308 120332 121356 122380
123404 124428 125452 126476 127500 128524 129548 130572 26680 26696 34179 26676 34691
67208 67209 6698 6702 26824 26840 26856 26872 6722 37251 26920 26936 26952
107872 38531 96308 27016 27032 27048 27064 29197 30221 31245 129076 6790 67234
67235 6802 6806 6810 40461 6818 6822 6826 6830 6834 6838 6842 6846
6850 109664 6858 6862 6866 109920 6874 56845 6882 6886 6890 6894 6898
6902 6906 6910 6922 27704 50307 71181 72205 73229 27784 27800 27816 27832
52355 79373 27880 27896 82445 111712 27944 27960 86541 87565 92214 96310 28040
28056 28072 28088 116790 120886 124982 97805 7046 7050 7054 7058 7062 7066
7070 7074 7078 7082 7086 7090 7094 7098 7102 7106 7110 7114 7118
7122 7126 7130 7134 7138 7142 7146 7150 7154 7158 7162 7166 67459
28808 28824 115360 115424 28872 115552 115616 115680 16910 17934 18958 115936 21006
71043 23054 96312 116256 72067 116384 116448 29198 30222 31246 129080 74115 116896
116960 117024 117088 117152 40462 117280 117344 117408 117472 117536 117600 117664 117728
29448 50702 117920 117984 78979 54798 55822 96313 29576 118368 7402 7406 118560
118624 29672 7422 118944 119008 83075 83331 26682 73230 29832 119392 119456 119520
85123 79374 119712 119776 82446 83470 86659 120032 84026 87566 92218 89614 120352
91662 120480 120544 116794 120890 124986 89987 90499 120992 121056 91267 103950 121248
92035 121376 121440 121504 121568 121632 121696 121760 121824 115214 116238 117262 122080
119310 95619 95875 122382 123406 122464 96899 97155 127502 97667 97923 130574 123168
99715 26684 7695 123424 123488 123552 123616 123680 123744 7738 7742 30984 102787
124064 124128 124192 22031 23055 96316 124448 124512 124576 124640 29199 30223 31247
129084 125024 7818 125152 125216 125280 125344 40463 125472 125536 125600 125664 125728
125792 7866 7870 125984 50703 126112 126176 126240 126304 126368 96317 126496 126560
126624 126688 126752 126816 7930 7934 127136 127200 115843 71183 26686 73231 127520
127584 127648 127712 78351 79375 127904 127968 128032 83471 128160 128224 84030 87567
92222 89615 128544 128608 128672 128736 116798 120894 124990 97807 99855 129184 129248
102927 103951 129440 105999 129568 129632 129696 129760 129824 129888 129952 130016 130080
116239 117263 130272 130336 130400 130464 122383 130592 130656 130720 130784 130848 130912
130976 131040 67650 115216 33826 4235 26690 4239 16914 37640 67658 115218 33830
4251 67662 127108 99368 75842 8526 73994 99370 92226 99371 8546 99372 5793
8558 73995 99374 106763 99375 24844 8586 8590 73996 99378 8602 123148 6689
24845 67690 67691 73997 99382 8634 8638 99384 75843 29572 73998 90382 121360
123150 123408 99388 125456 126480 73999 90383 106767 123151 16962 17090 69840 17468
17442 70408 82466 94344 17986 27682 29730 127112 18626 37922 28809 19010 19138
19266 70384 50210 52258 82470 94345 4409 4411 4413 17660 20674 70690 28810
21058 21186 78882 70896 69770 85026 86154 115242 88328 88840 118922 127114 22722
103458 28811 92424 23234 53387 71408 115746 117794 86155 115246 123938 97032 128034
130082 5155 67779 33890 99426 33891 99427 17443 103176 86156 115250 72080 4507
29731 127116 26818 26946 115252 18084 4523 72400 18108 50211 27842 86157 115254
18148 18156 18164 18172 28866 70691 115256 29250 117512 78883 118536 69774 119560
82490 115258 121096 121608 118926 127118 123656 124168 115260 125192 125704 53391 126728
115747 127752 82494 115262 129288 129800 130312 130824 19492 82498 67851 25636 27684
82499 67855 99464 18962 37176 33930 99466 92232 123170 16966 18636 67866 67867
33934 99470 125000 129096 4675 82504 115272 74896 18732 82505 18748 16970 21257
82506 115274 38466 93220 82507 115275 99480 78468 19681 33946 99482 55826 99483
16974 29449 67898 115278 30985 125988 67902 115279 67906 115280 33954 99490 26698
99491 16978 37641 67914 115282 4761 4763 67918 115283 24874 84498 21729 74026
99498 92234 123178 9570 24875 88708 9582 74027 99502 125002 129098 4803 82520
115288 4809 4811 82521 19260 9634 4819 67946 115290 33974 4827 67950 115291
99512 75851 95108 33978 99514 121362 123182 123410 99516 96900 126482 33982 99518
125003 129099 19500 78032 78064 69784 70409 82530 115298 71945 27686 118936 127128
51394 82532 28825 78480 4907 78544 78576 69785 52262 82534 115302 19684 19692
19700 19708 53442 82536 115304 53826 53954 53402 79088 17002 85030 82538 115306
88329 88841 118938 127130 55490 82540 115308 92425 56002 53403 79600 17006 117798
82542 115310 123942 97033 128038 130086 28833 28897 74040 99554 26702 123192 116868
24889 117380 117636 74041 99558 40168 40184 99560 84499 119684 34026 99562 92238
96334 30241 99564 30369 121732 34030 99566 106811 123195 99568 123524 40504 34034
90428 124548 123196 31265 99572 125572 40632 34038 90429 126596 40696 99576 75855
127876 34042 90430 121363 96335 129156 99580 129668 129924 34046 90431 106815 123199
34058 99594 34059 99595 17030 68802 86177 115334 5177 5179 5181 20732 69826
17674 28834 70210 19210 53410 83184 69794 85032 86178 115338 22794 93224 95272
97320 20876 103464 28835 26890 27402 27914 83696 69795 117800 86179 115342 123944
125992 128040 130088 28836 21028 21036 84176 21052 17042 37642 86180 115346 5273
21100 118948 127140 21132 76098 28837 21156 5291 21172 21180 17046 45834 86181
115350 21220 21228 21236 21244 78018 50442 115352 78402 51978 53414 85232 17050
54026 82586 115354 55562 56074 95273 97321 99640 41294 40161 34106 90446 34107
99643 100435 99644 125460 112723 34110 90447 34111 99647 82626 86224 21564 17450
19498 82594 94376 71946 27690 118952 127144 84162 37930 115364 84546 76554 86736
86768 69801 52266 82598 94377 5433 5435 5437 5439 82698 70698 115368 84234
84746 53418 87280 69802 85034 82602 89130 91178 88842 95274 97322 90890 91402
115372 107562 92938 111658 87792 115754 94986 82606 121898 123946 97034 128042 98058
115376 22052 22060 88272 22076 17451 103178 82610 94380 88464 88496 118956 127148
107274 107786 115380 88720 88752 88784 88816 69805 111370 82614 94381 5561 22252
22260 22268 115466 70699 115384 117002 94914 53422 118538 69806 119562 82618 89131
121098 121610 95275 97323 123658 124170 115388 96834 96962 126218 126730 115755 127754
82622 121899 129290 129802 130314 130826 115394 25644 6923 118960 127152 24930 41314
11342 74082 99722 34187 99723 11362 24931 11370 11374 74083 99726 34191 99727
5699 17675 115400 102978 19211 53426 91376 69810 85036 82634 94386 22795 93228
118962 127154 5731 103468 115404 5737 22956 22964 22972 69811 117804 82638 94387
123948 125996 128044 130092 115408 23076 5771 92368 92400 69812 37643 82642 115410
92560 92592 118964 127156 5795 23188 115412 23204 23212 23220 23228 69813 23244
82646 115414 23268 23276 23284 23292 93232 50443 115416 111170 93360 53430 93424
17114 54027 82650 115418 55563 56075 118966 127158 99768 41326 56545 34234 90478
34235 99771 123414 99772 96901 126486 34238 90479 34239 99775 115522 94448 17122
19502 82658 94392 71947 27694 118968 127160 116930 37934 115428 76043 117442 117570
94960 17126 52270 82662 94393 95120 95152 118594 95216 82699 70702 115432 84235
84747 53434 95472 17130 85038 82666 94394 91182 88843 95278 97326 121026 91403
115436 107566 92939 111662 95984 17134 94987 82670 94395 123950 97035 128046 98059
115440 123458 96432 123714 96496 17138 103179 82674 115442 96656 96688 118972 127164
125122 107787 115444 96912 96944 125762 97008 17142 111371 82678 115446 126530 126658
126786 97264 127170 70703 115448 127554 127682 53438 97520 17146 119563 82682 94398
121099 121611 95279 97327 129218 129346 115452 125195 125707 126219 98032 17150 127755
82686 94399 129291 129803 130315 130827 24972 20675 28867 6723 25004 27916 25020
69827 117808 82702 115470 123952 126000 82703 130096 50232 74120 90504 26722 123272
50312 70241 70305 70369 74121 90505 12602 12606 24970 84504 70881 74122 99882
106890 96354 71201 24971 71329 71393 74123 99886 125026 123275 12682 12686 74124
99890 72097 123276 50824 108056 12714 12718 74125 99894 12730 12734 24974 117272
72929 34362 99898 106894 96355 123416 24975 125464 126488 34366 99902 129560 123279
25652 25660 69832 78024 82722 115490 17987 27698 118984 127176 25740 82724 28873
6441 6443 103120 25788 69833 52274 82726 115494 25828 25836 25844 25852 82728
115496 21059 21187 53450 6479 17194 78026 86218 94410 91186 93234 118986 127178
25996 82732 115500 107570 26028 111666 26044 17198 117810 86219 94411 96524 126002
118987 130098 68803 74136 99938 26726 123288 78369 51590 13098 13102 74137 99942
13114 13118 24986 84505 79073 34410 99946 106906 123290 55430 24987 79521 52664
34414 99950 125030 123291 82744 115512 29251 117516 53454 118540 17210 78030 86222
115514 121100 121612 82747 127182 123660 82748 115516 125196 125708 126220 126732 17214
117811 86223 115518 129292 129804 82751 130099 110800 118992 127184 99976 6693 83169
34442 99978 34443 99979 69841 26828 86225 115526 6713 26860 26868 26876 20690
115528 37443 26924 53458 26940 17226 78034 82762 94418 38467 110802 118994 127186
27020 20691 115532 6761 27052 27060 27068 17230 78035 82766 94419 123956 110803
118995 127187 86241 74152 90536 26730 123304 86561 54424 54440 54456 74153 55402
13626 13630 71786 75882 87265 34474 90538 106922 123306 88198 104554 54696 54712
34478 120938 125034 129130 68962 68963 34482 90540 54888 123308 54920 92550 13738
13742 34486 55403 13754 13758 71787 117274 89313 34490 90542 106926 123310 123418
124442 125466 126490 34494 120939 129562 130586 27708 69848 78040 82786 94424 71949
27702 119000 127192 27788 37942 115556 27812 27820 27828 27836 69849 52278 82790
94425 27876 27884 27892 27900 20698 115560 6985 27948 53466 27964 17258 78042
82794 115562 91190 93238 119002 127194 28044 20699 115564 107574 28076 111670 28092
17262 78043 82798 115566 96525 126006 119003 127195 94433 34530 100066 26734 123320
94753 117126 94881 94945 34534 100070 14138 14142 71790 75886 95457 34538 100074
106938 123322 95777 104558 95905 95969 34542 100078 125038 129134 96417 96481 34546
100082 96673 123324 96801 108059 96929 96993 34550 100086 97185 97249 71791 117275
97505 34554 100090 106942 123326 97825 124443 97953 98017 34558 100094 130694 130587
115600 115632 115664 115696 20706 115592 70211 19214 53474 115952 83000 85048 87096
89144 22798 93240 95288 97336 116272 25870 115596 26894 27406 27918 116464 115768
117816 119864 121912 123960 126008 128056 130104 29220 116912 116944 116976 17465 78052
86244 94436 117136 117168 119012 127204 117296 76099 115604 117392 117424 117456 117488
50233 45838 86245 94437 117648 117680 117712 117744 20710 115608 78403 51982 53478
118000 83001 54030 87097 89145 55566 56078 95289 97337 29580 29588 115612 29604
7403 7405 7407 118544 118576 118608 118640 7417 29676 29684 7423 119024 17466
78056 86248 94440 71950 72462 119016 127208 119344 37946 115620 84547 76558 119504
119536 69865 52282 86249 94441 119696 119728 119760 119792 20714 115624 84238 84750
53482 120048 69866 86798 86250 94442 91194 88846 95290 97338 90894 103482 115628
107578 92942 111674 120560 115770 94990 119866 121914 96526 97038 128058 98062 120976
121008 121040 121072 17467 78060 86252 94444 121232 121264 119020 127212 107278 107790
115636 30372 121520 121552 121584 69869 111374 86253 94445 121744 121776 121808 121840
20718 115640 117006 94915 53486 118542 69870 119566 86254 94446 121102 121614 95291
97339 123662 124174 115644 96835 96963 126222 126734 115771 127758 119867 121915 129294
129806 128059 130830 127216 123440 123472 115652 123536 123568 123600 123632 69873 123696
54332 94449 30948 7739 7741 7743 20722 115656 102979 19215 53490 124144 69874
85052 87100 89148 22799 93244 95292 97340 124464 25871 115660 124560 124592 124624
124656 115772 117820 119868 121916 123964 126012 128060 130108 31268 7819 125136 125168
17469 78068 21565 94452 125328 125360 119028 127220 125488 125520 115668 125584 125616
125648 125680 69877 125744 54333 94453 31460 7867 31476 7871 20726 115672 126096
126128 53494 126192 69878 54031 87101 89149 55567 56079 95293 97341 126512 126544
115676 126608 126640 126672 126704 126736 113859 126800 126832 7929 7931 7933 7935
17470 78072 86264 94456 71951 72463 119032 127224 127536 37950 115684 76047 127664
127696 127728 50238 52286 86265 94457 127888 127920 127952 127984 20730 115688 84239
84751 53498 128240 69882 86799 87102 94458 91198 88847 95294 97342 128560 103486
115692 107582 92943 111678 128752 115774 94991 119870 121918 96527 97039 128062 98063
129168 129200 129232 129264 17471 78076 86268 94460 129424 129456 119036 127228 129584
107791 115700 129680 129712 129744 129776 50239 111375 86269 94461 129936 129968 130000
130032 20734 115704 130192 130224 53502 130288 69886 119567 87103 94462 121103 121615
95295 97343 130608 130640 115708 125199 125711 126223 130800 115775 127759 119871 121919
129295 129807 128063 130831 17506 18978 69944 21026 71048 71304 96392 26146 70056
70072 29218 120968 73352 129160 17546 70200 74888 18786 17562 123428 19042 70312
19170 19234 19298 17594 17598 78216 78472 19682 78984 79240 79496 96393 69945
17642 17646 17650 17654 17658 17662 70178 22666 83592 123432 21090 21154 70840
78370 55434 17722 70904 86408 86664 21730 86562 88202 87688 123434 88456 88712
71096 89224 89480 89736 97826 22690 22754 22818 91528 22946 123436 92552 71336
71352 93320 93576 17850 17854 94600 94856 95112 95368 95624 95880 123438 96648
96904 97160 97416 97672 97928 98184 50274 70026 70027 17507 50275 70030 70031
25698 71976 71992 103560 103816 104072 123442 26210 18026 18030 29219 120972 106120
129164 18058 26850 26914 26978 27042 123444 70057 18090 18094 18098 18102 18106
18110 27746 72488 72504 27938 28002 28066 123446 70073 70074 70075 70076 70077
18170 18174 115848 22670 116360 123448 117128 117384 117640 78371 55438 72936 72952
119176 119432 119688 86563 88206 120456 96398 121224 30370 121736 121992 122248 122504
97827 123528 73272 124040 124296 124552 123452 125320 125576 73400 126088 126344 126600
73464 127368 127624 127880 128136 128392 128648 96399 129416 129672 129928 130184 130440
130696 130952 17545 50313 70182 37180 50314 70186 115850 17547 50315 70190 115851
35938 18634 18638 18642 36194 18650 123462 70201 18666 18670 18674 18678 18682
18686 71204 18057 123464 74904 18730 18734 78372 55442 70222 70223 20873 75048
75064 86564 88210 88612 123466 75160 75176 75192 94756 95780 96804 97828 83096
115864 70244 70245 83097 19697 50330 83098 115866 17563 50331 83099 115867 116260
117284 40162 29833 120356 121380 123470 124452 125476 126500 127524 128548 129572 130596
18982 70282 70283 70284 70285 19002 19006 76056 19018 19022 76104 42338 38537
123474 19046 19050 19054 29221 120980 31269 129172 83112 115880 70308 70309 83113
38204 70313 70314 115882 70316 70317 70318 115883 76568 19146 19150 19154 19158
19162 123478 19174 19178 19182 19186 19190 19194 19198 71205 50825 123480 19238
19242 19246 78373 55446 70350 70351 19270 19274 19278 86565 88214 54921 123482
19302 19306 19310 94757 95781 96805 97829 83128 115896 70372 70373 83129 95172
70377 83130 115898 17595 70381 83131 115899 70385 83132 115900 70388 70389 83133
126738 50366 83134 115902 17599 50367 83135 115903 50402 84120 71049 71305 96408
26150 78248 50914 116888 120984 73353 129176 78376 51426 51490 78424 78440 123492
78488 19626 19630 78536 78552 78568 78584 78217 78632 78648 78985 79241 79497
96409 19686 19690 19694 19698 19702 19706 19710 22682 83593 123496 79000 79016
79032 51354 55450 79080 79096 86409 86665 79160 84122 88218 87689 123498 88457
79272 79288 89225 89481 89737 129178 79400 55522 79432 55650 79464 123500 79512
79528 79544 79560 79576 79592 79608 94601 94857 95113 95369 95625 95881 123502
96649 96905 97161 97417 97673 97929 98185 70538 94520 25678 27726 70542 127288
12601 70546 115940 42062 44110 70550 117700 50406 70554 94521 40164 40172 70558
29681 70562 115944 37178 45370 70566 119748 50410 70570 94522 102714 110906 70574
127290 12603 70578 115948 37179 45371 70582 121796 50414 70586 94523 102715 110907
70590 127291 123588 83185 40508 50418 70602 94524 124484 124612 70606 127292 12605
70610 115956 125508 125636 70614 31473 50422 70618 94525 58447 60495 70622 40700
83192 115960 37182 45374 83193 127940 50426 70634 94526 102718 110910 70638 127294
12607 70642 115964 37183 45375 70646 129988 50430 70650 94527 102719 110911 70654
127295 83496 70818 83256 86568 87592 88616 89640 22922 23178 83384 94760 95784
96808 97832 83512 25738 25994 72098 123532 72290 83624 72418 72482 28042 83688
83704 116264 117288 72930 119336 120360 121384 122408 124456 125480 126504 127528 128552
129576 130600 70794 70795 84168 84184 70798 70799 37258 21066 84280 75042 38282
38538 96420 21094 21098 21102 116900 120996 31273 129188 70819 76066 84568 21146
123540 21158 21162 21166 21170 84696 21178 21182 45450 84776 84792 21202 46474
21210 96421 70841 70842 70843 21234 21238 70846 70847 22694 72233 123544 51594
85160 85176 78377 79401 85224 85240 53642 53898 85304 86569 87593 88617 89641
55690 55946 85432 94761 95785 96809 97833 116024 70884 70885 83257 118548 70889
70890 116026 17723 70893 70894 116027 70897 70898 116028 70900 70901 70902 61775
50494 70906 116030 70908 50495 70910 116031 22058 23082 96424 26154 86440 86456
116904 121000 73354 129192 84194 84258 75146 86632 123556 86680 86696 86712 84770
86744 86760 86776 50730 86824 86840 53802 54826 55850 96425 21734 70970 70971
70972 70973 70974 70975 22698 72234 123560 87192 87208 87224 51370 55466 87272
87288 86410 84522 87352 84138 88234 87690 96426 91690 87464 87480 94762 95786
96810 97834 88290 91274 88418 87656 123564 92554 87720 87736 93322 88930 87784
87800 94602 117290 95114 119338 120362 121386 122410 96650 96906 97162 97418 97674
97930 130602 22062 88264 88280 22074 22078 102794 88360 88376 103562 103818 104074
96428 88472 88488 88504 116908 121004 106122 129196 88632 107658 92514 88680 123572
88728 88744 88760 88776 88792 88808 88824 110986 88872 88888 88904 93538 93602
96429 71097 83310 71099 71100 71101 83311 71103 22702 116362 123576 94818 89256
89272 51374 55470 89320 89336 119178 119434 89400 84142 88238 120458 96430 95842
89512 89528 121994 122250 122506 122762 89656 124042 96610 89704 123580 125322 89768
89784 97058 126346 89832 89848 127370 97442 127882 128138 128394 128650 122411 129418
97954 98018 130186 130442 130698 130954 22706 22710 22714 22718 12683 22730 22734
22738 101730 22746 96433 22758 22762 22766 22770 22774 22778 22782 18059 123592
91288 91304 91320 51378 55474 91368 91384 83500 103586 91448 84146 88242 88620
123594 91544 91560 91576 94764 95788 96812 97836 22926 22930 91736 22938 123596
22950 22954 22958 22962 91864 71278 71279 116268 117292 105698 119340 120364 121388
123598 124460 125484 126508 127532 128556 129580 130604 23086 92360 92376 92392 92408
37259 92456 92472 92488 107874 38539 96436 92568 92584 92600 116916 121012 31277
129204 23182 23186 23190 23194 123604 71337 71338 71339 23218 23222 71342 71343
109666 23242 23246 23250 109922 23258 96437 71353 83374 71355 71356 71357 83375
71359 72237 123608 93336 93352 93368 51382 55478 93416 93432 111714 93480 93496
84150 88246 88621 123610 93592 93608 93624 94765 95789 96813 97837 116152 71396
71397 83385 95173 50618 83386 116154 17851 50619 83387 116155 71409 71410 116156
71412 71413 71414 126742 71417 71418 116158 17855 71421 71422 116159 96440 26158
94632 94648 116920 121016 73355 129208 94776 94792 117090 94824 123620 94872 94888
94904 94920 94936 94952 94968 50734 95016 95032 53806 54830 55854 96441 95128
95144 95160 95176 95192 95208 95224 72238 123624 95384 95400 95416 51386 55482
95464 95480 86411 84526 95544 84154 88250 87691 123626 91694 95656 95672 94766
95790 96814 97838 95800 91275 91531 95848 123628 95896 95912 95928 121634 121698
95976 95992 94603 117294 96056 119342 120366 121390 123630 96152 96907 97163 97419
97675 97931 130606 96456 96472 96488 96504 102795 96552 96568 124194 103819 104075
96444 96664 96680 96696 116924 121020 106123 129212 96824 125218 125282 96872 123636
96920 96936 96952 96968 96984 97000 97016 110987 97064 97080 97096 126306 126370
96445 97176 97192 97208 97224 97240 97256 97272 116363 123640 97432 97448 97464
51390 55486 97512 97528 119179 128162 97592 84158 88254 120459 96446 97688 97704
97720 121995 122251 122507 122763 97848 124043 124299 97896 123644 97944 97960 97976
97992 129890 98024 98040 127371 127627 98104 130338 130402 130466 96447 98200 98216
98232 98248 130914 130978 131042 35939 12687 50738 83506 116274 72081 72113 119180
116275 20877 71891 12713 101482 12717 123674 50742 83510 71899 12729 101486 12733
123675 116280 37262 45454 53646 118552 50746 86414 94606 102798 110990 119182 127374
83516 116284 37263 125720 111715 126744 50750 86415 94607 102799 110991 119183 127375
26162 25706 25710 116936 121032 125128 129224 25742 74892 18787 25754 123684 71977
25770 25774 103112 103128 25786 25790 78476 25806 53810 54834 55858 96457 71993
83534 71995 71996 71997 83535 71999 26826 123688 103576 103592 103608 51402 55498
25914 25918 84530 103736 84170 88266 92362 123690 91698 88716 103864 116938 121034
125130 129226 25998 22819 91532 26010 123692 92556 72042 72043 93324 93580 72046
72047 117298 95116 119346 88267 121394 123694 96652 96908 97164 127538 128562 97932
130610 26166 26170 26174 26186 26190 74546 88268 92364 123698 26214 26218 26222
74547 121036 125132 123699 116328 37274 72101 72102 79089 72105 72106 116330 102810
72109 72110 127386 72114 116332 37275 72117 72118 52668 72121 72122 116334 102811
72125 72126 127387 26830 123704 117132 117388 117644 51406 55502 105704 105720 119436
119692 84174 88270 92366 96462 121228 30371 121740 116942 121038 125134 129230 106040
124044 124300 124556 123708 125324 125580 106168 126092 126348 126604 106232 127628 127884
119347 88271 128652 96463 129420 129676 129932 127539 128563 130700 130611 26834 26838
26842 123718 26854 26858 26862 72252 72253 72254 72255 123720 107672 72266 72267
51410 55506 26938 26942 26954 26958 84178 88274 92370 123722 107928 26986 26990
116946 121042 125138 129234 72291 27026 108120 27034 123724 27046 27050 27054 27058
108248 72302 72303 117300 40163 84179 88275 121396 123726 124468 125492 126516 116947
121043 129588 129235 27754 119208 127400 83620 116388 76058 76570 86737 86769 78249
86441 94633 13625 13627 13629 13631 116392 37290 76906 53674 87281 50858 86442
94634 102826 111018 119210 127402 83628 116396 107626 109674 53675 87793 50862 86443
94635 124010 126058 119211 130154 123736 50866 72394 72395 88465 88497 72398 72399
72402 72403 36202 101738 92374 123738 50870 72410 72411 36206 101742 125142 123739
116408 37294 76907 72422 118554 50874 72426 94638 102830 111022 72430 127406 72434
116412 125210 125722 72438 126746 50878 72442 94639 129306 129818 72446 130842 27758
116952 121048 125144 129240 72483 111176 111192 27802 123748 72489 27818 27822 27826
27830 72494 72495 27850 27854 53814 54838 55862 96473 72505 72506 72507 72508
72509 72510 72511 123752 27942 72522 72523 51418 55514 27962 27966 84534 27982
84186 88282 92378 123754 91702 28010 28014 116954 121050 125146 129242 28046 112200
55651 28058 123756 112280 28074 28078 112328 112344 72558 72559 117302 95117 84187
88283 121398 123758 96653 96909 97165 116955 121051 97933 129243 119224 127416 83684
116452 76059 117446 117574 94961 50918 86457 94649 56548 14139 118598 14143 116456
37306 45498 53690 95473 50922 83690 94650 102842 111034 119226 97390 83692 116460
107630 109678 53691 95985 50926 83694 94651 124014 126062 119227 130158 96497 50930
86460 94652 96657 96689 119228 127420 83700 116468 96913 96945 109851 97009 50934
86461 94653 112923 126662 126790 97265 116472 37310 45502 53694 97521 50938 86462
94654 102846 111038 119230 97391 83708 116476 125211 125723 53695 98033 50942 86463
94655 129307 129819 119231 130843 117304 72931 119352 120376 121400 122424 124472 125496
126520 127544 128568 129592 130616 116968 116984 117032 117048 75043 38286 92388 96484
117144 117160 117176 116964 121060 125156 129252 117320 117336 117352 40505 117400 117416
117432 117448 117464 117480 117496 117544 117560 117576 46478 117608 96485 117656 117672
117688 117704 117720 117736 117752 73273 51598 117928 117944 51430 55526 117992 118008
53902 118072 84198 88294 88633 89657 55694 55950 118200 94777 95801 96825 97849
72932 72933 118376 118392 72937 29610 29614 72940 72941 72942 72943 118568 118584
118600 118616 118632 118648 72953 72954 72955 29682 29686 72958 72959 121064 125160
129256 119368 75150 119400 40506 119448 119464 119480 119496 119512 119528 119544 119592
119608 53818 54842 55866 96489 119704 119720 119736 119752 119768 119784 119800 73274
119960 119976 119992 51434 55530 120040 120056 84538 120120 84202 88298 87694 89658
91706 120232 120248 89230 95802 96826 129258 91278 88419 120424 106042 92558 120488
120504 93326 88931 120552 120568 94862 95118 119354 95630 95886 122426 96654 96910
97166 97422 97678 97934 98190 121080 121128 121144 103566 22075 92396 96492 121240
121256 121272 116972 121068 125164 129260 107662 92515 121448 40507 121496 121512 121528
121544 121560 121576 121592 121640 121656 121672 93539 121704 96493 121752 121768 121784
121800 121816 121832 121848 73275 94819 122024 122040 51438 55534 122088 122104 119438
122168 84206 88302 120462 89659 95843 122280 122296 121998 95803 96827 129262 124046
96611 122472 106043 125326 122536 122552 97059 126350 122600 122616 97443 127886 128142
128398 128654 122427 129422 97955 98019 130190 130446 130702 130958 73276 73277 73278
73279 124056 124072 124088 51442 55538 124136 124152 124200 124216 84210 88306 88636
96498 124312 124328 124344 94780 95804 96828 97852 124488 124504 124520 106044 124568
124584 124600 124616 124632 124648 124664 117308 124728 119356 120380 121404 122428 124476
125500 126524 127548 128572 129596 130620 125176 125224 125240 125256 107875 92404 96500
125336 125352 125368 116980 121076 125172 129268 125512 125528 125544 40509 125592 125608
125624 125640 125656 125672 125688 125736 125752 125768 125784 125800 96501 73401 73402
73403 73404 73405 73406 73407 126104 126120 126136 51446 55542 126184 126200 126248
126264 84214 88310 88637 96502 126360 126376 126392 94781 95805 96829 97853 126536
126552 126568 126584 126616 126632 126648 126664 126680 126696 126712 126760 126776 126792
126808 126824 126840 73465 73466 73467 73468 73469 73470 73471 129272 127560 127576
127592 40510 127640 127656 127672 127688 127704 127720 127736 127784 127800 53822 54846
55870 96505 127896 127912 127928 127944 127960 127976 127992 128152 128168 128184 51450
55546 128232 128248 84542 128312 84218 88314 87695 89662 91710 128424 128440 89231
95806 96830 97854 91279 91535 128616 106046 128664 128680 128696 128712 128728 128744
128760 94863 128824 119358 95631 95887 122430 128920 96911 97167 97423 97679 97935
98191 129320 129336 129352 22079 92412 96508 129432 129448 129464 116988 121084 125180
129276 129608 129624 129640 40511 129688 129704 129720 129736 129752 129768 129784 129832
129848 129864 129880 129896 96509 129944 129960 129976 129992 130008 130024 130040 130200
130216 130232 51454 55550 130280 130296 130344 130360 84222 88318 120463 89663 130456
130472 130488 121999 95807 96831 97855 124047 124303 130664 106047 130712 130728 130744
130760 130776 130792 130808 127631 130872 130888 130904 130920 122431 130968 130984 131000
131016 131032 131048 131064 38034 117028 18731 37451 18735 51494 84262 94793 18747
37455 18751 84772 53834 38130 51498 84266 117034 111178 119370 127562 21067 117036
75180 93476 75196 51502 84270 117038 126098 119371 130194 51602 19625 78564 79433
19689 19705 76070 79076 91290 107674 124058 55698 79524 79588 91291 107675 124059
117540 53838 118564 51514 84282 94798 111182 119374 127566 84284 117052 125732 126244
126756 51518 84286 94799 129828 119375 130852 91300 38042 124068 75049 38058 38062
38070 38074 38078 85156 85220 88361 92457 96553 75065 38122 38126 38134 38138
38142 86628 86692 86756 55594 38202 21753 76074 87268 91306 107690 124074 91722
21929 87780 121130 125226 129322 91308 88484 124076 92562 88740 88804 112202 38330
38334 117322 89316 91310 107694 124078 96658 125514 126538 121131 129610 129323 91364
103114 107698 103115 78425 75162 75163 103118 125228 103119 103122 92580 103123 51562
84330 117098 111194 84331 127578 93348 93412 103130 107702 103131 51566 75194 117102
111195 75198 127579 94820 94884 94948 55598 95140 95204 76078 95460 91322 107706
96558 121234 95908 95972 121134 125230 129326 91324 96676 124092 125330 96932 96996
112203 97188 97252 117323 97508 91326 107710 96559 129426 97956 98020 121135 129611
129327 51634 119396 127588 78484 117140 19627 19629 19631 51606 54425 94821 78572
78580 78588 78537 79001 52466 51610 87193 117146 79561 95385 97433 19685 105700
103226 19693 103227 124492 19701 112947 103230 19709 103231 117156 53938 53970 54002
78441 86633 94825 79084 79092 79100 84774 53866 54514 51626 86634 94826 93338
95386 97434 103578 117164 54962 88905 55026 51630 119962 122010 126106 128154 130202
86636 94828 55730 119404 127596 92489 117172 55986 56018 56050 51638 86637 94829
79596 79604 79612 94921 53870 95177 51642 86638 94830 121638 95387 97435 124198
117180 96969 97097 97225 51646 119963 122011 97993 130342 98249 124132 117348 117412
117476 117604 117668 117732 117924 117988 88377 92473 96569 118372 29609 29625 118628
29673 40190 119460 119524 55610 119716 119780 76090 120036 88378 107754 96570 91726
120484 120548 95822 96846 129338 91372 121252 124140 121444 121508 121572 55611 121764
121828 76091 122084 88379 107758 96571 122468 125518 126542 121147 125243 129339 88380
107762 96572 124516 124580 124644 121148 125244 129340 91380 125348 124148 125540 125604
125668 125796 31465 40638 126116 126180 88381 107766 96573 126564 126628 126692 126820
40698 40702 127652 127716 55614 127908 127972 76094 128228 88382 107770 96574 128612
128676 128740 95823 96847 129342 91388 129444 124156 129636 129700 129764 55615 129956
130020 76095 130276 88383 107774 96575 130660 130724 130788 121151 125247 129343 76106
76107 103590 42298 42302 21733 91434 107818 124202 103596 88724 42350 91435 125258
129354 84780 84569 84796 21203 76138 117338 21211 76142 117339 95124 91438 107822
124206 103612 125522 126546 91439 125259 130642 85164 78546 85180 78633 86681 94873
85228 85236 85244 53962 53914 85308 86825 84586 94874 88873 119450 127642 117356
56010 56138 85436 95017 84590 94875 97065 128166 127643 117396 117652 91449 40169
40185 119700 103658 107834 124218 25915 30373 121748 103662 107835 129358 103666 124564
124220 39247 125588 40633 103670 126612 40697 127892 103674 107838 124222 25919 129684
129940 103678 107839 130643 86700 84690 84722 45866 54441 94885 86764 86772 86780
52010 79017 85234 54058 87209 89257 56106 95401 97449 40165 103738 38203 103739
103740 38205 58703 103742 21757 103743 84682 84810 87228 52394 86697 94889 87276
87284 87292 84778 53930 87356 86826 87210 89258 87754 95402 97450 117420 88778
93482 87484 95018 119978 122026 97066 128170 130218 94892 91850 119468 127660 117428
87724 88786 87740 93386 86701 94893 87788 87796 87804 117546 53934 118570 119594
87211 89259 121642 95403 97451 117436 96970 126250 97226 127786 119979 122027 129834
130346 130858 117450 93356 119474 97452 117452 88492 88500 91890 117932 84686 117454
126124 128172 130220 94900 92594 119476 127668 117460 88748 88756 88764 88780 84694
94901 88812 88820 88828 88876 53942 88892 54059 84698 117466 112330 84699 97453
56549 103866 107886 103867 103868 125526 112987 103870 107887 103871 117578 89276 52398
86713 94905 89324 95186 89340 119498 53946 89404 86827 84714 94906 88875 95406
97454 117484 121546 121674 89532 95019 84718 122030 97067 128174 130222 94908 89708
119484 127676 117492 96946 125770 89788 111403 84726 94909 126666 126794 89852 127690
53950 89916 119595 84730 94910 128714 95407 97455 117500 129738 129866 129994 127787
84734 122031 130762 130890 131018 117548 104114 93484 104178 52014 84782 117550 97068
119499 98092 45870 91545 45882 45886 79077 104042 107930 124314 91737 79525 45934
104046 107931 124315 53966 118572 52026 86734 94926 121644 119502 127694 117564 125740
126252 126764 52030 86735 94927 129836 119503 130860 93364 119506 127698 117580 92588
92596 92604 78547 84814 117582 126132 119507 130228 54457 91561 46394 46398 87269
104106 107946 124330 104810 54697 54713 104110 125290 124331 54889 124332 92566 46506
46510 104118 46522 46526 89317 104122 107950 124334 124506 125530 126554 104126 129626
124335 93372 52406 86745 94937 93420 93428 93436 53978 93500 78554 84842 117610
88877 119514 127706 117612 93612 56139 93628 78555 84846 117614 97069 119515 98093
94949 91577 56553 56569 95461 104170 107962 96622 104814 95909 95973 104174 125294
124347 96677 124348 125334 96933 96997 104182 97189 97253 97509 104186 107966 96623
124507 97957 98021 104190 129627 124351 94956 94964 94972 79033 95036 54062 87225
89273 56110 95417 97465 117660 95148 95156 95164 95180 95188 118642 95212 95220
95228 95420 52410 86761 94953 95468 95476 95484 53994 95548 85178 87226 87854
87755 95418 97466 117676 88779 93486 95676 95022 119994 122042 97070 128186 130234
95852 119532 127724 117684 95916 121554 95932 93387 86765 94957 95980 95988 95996
53998 118574 85179 87227 89275 121646 95419 97467 117692 96971 126254 97227 127790
119995 122043 129838 130350 130862 97468 117708 96684 96692 96700 117948 119996 122044
126140 128188 130236 96876 119540 127732 117716 96940 125650 96956 125746 86773 94965
97004 97012 97020 54006 97084 54063 87229 89277 112331 95421 97469 117724 97196
126674 97212 113867 126802 126834 97260 97268 97276 52414 86777 94969 97516 97524
97532 54010 128242 85182 87230 87855 88879 95422 97470 117740 121547 121675 97724
95023 119998 122046 97071 128190 130238 97900 119548 127740 117748 97964 129746 97980
111407 86781 94973 98028 98036 98044 54014 98108 85183 87231 89279 128715 95423
97471 117756 129739 129867 98236 127791 119999 122047 130763 130891 131019 78649 88473
79513 96665 52458 52462 52470 52474 52478 55706 79081 79097 79161 91754 92570
96666 79273 79289 121242 125338 129434 79465 124524 79529 79545 79577 78702 78703
95129 91758 95897 96153 96921 97177 97689 97945 98201 86841 95033 105708 118604
105724 54074 119756 85226 95034 111418 95438 97486 117996 121548 54075 121804 85230
95035 111419 119611 127803 124620 119612 127804 118004 125644 125772 106172 85238 95037
126668 113971 106236 54078 127948 86846 95038 111422 95439 97487 118012 129740 54079
113871 86847 95039 111423 119615 127807 85305 87657 88681 89705 55962 54126 95849
96873 97897 118580 85306 118074 79085 85307 118075 118076 79093 85309 62287 79098
118078 79101 79102 118079 88489 55914 96681 54506 54510 54518 54522 54526 55722
87273 87289 87353 88490 88682 96682 87465 87481 95850 96874 129450 124588 87721
87737 93594 87785 87801 95130 120426 95898 96154 96922 126570 128618 97946 130666
88505 121260 125356 129452 124596 54954 54958 54966 54970 54974 88889 88493 112282
96685 55018 55022 55030 55034 55038 55726 89321 89337 89401 88494 88683 96686
89513 89529 95851 122522 129454 124604 89769 89785 126362 89833 89849 127898 120427
128666 128922 129690 129946 130458 130714 130970 55734 55738 55742 118380 91854 121452
122476 125548 126572 128620 129644 130668 55918 121268 125364 129460 124628 55978 55982
55990 55994 55998 56014 88501 56026 96693 56042 56046 56054 56058 56062 56122
56126 56142 91866 92598 89709 56170 56174 95853 96877 97901 95181 85434 118202
52667 85435 118203 118204 125750 54127 62319 85438 118206 52671 85439 118207 96697
56554 56558 95193 56570 56574 95465 95481 95545 91882 88686 96698 95657 95673
95854 96878 97902 124652 95913 95929 121702 95977 95993 96057 91886 95899 96155
96923 126574 128622 97947 130670 121276 125372 129468 124660 96937 96953 96985 97001
97017 97081 91894 126374 96701 97193 97209 97241 97257 97273 97513 97529 97593
91898 88687 96702 97705 97721 95855 122523 97903 124668 97961 97977 129894 98025
98041 98105 91902 130470 128923 98217 98233 130918 130715 130971 119708 80335 105706
29613 124730 40173 121756 105710 29629 124731 118588 118612 118394 118636 118644 118395
105722 29677 124734 29685 113103 105726 40191 124735 119724 119732 119740 119764 119772
119788 119796 119804 120044 120052 120060 120124 87274 95146 88890 89402 89914 88782
93498 120252 95546 122090 126186 97594 98106 127916 120492 88787 120508 111930 56555
120556 120564 120572 118586 95566 95150 121658 122170 97515 96974 126266 97230 128314
128826 97998 130362 130874 122092 126188 128236 130284 127924 121516 121524 121532 46395
56557 121580 121588 121596 121660 87277 89325 56123 95469 97517 121772 121780 121788
121812 121820 121836 121844 121852 95187 122108 122172 87278 95162 88891 89403 89915
121550 93499 122300 95547 122094 126190 97595 98107 127932 122540 125774 122556 111931
56559 126670 126798 122620 122684 120123 95166 121659 122171 97519 129742 126267 129998
128315 128827 129851 130894 131022 125756 126268 126780 128316 118590 129852 130364 130876
126196 128244 130292 127956 125612 125620 125628 44367 95189 125676 125684 125692 87285
118618 125804 95477 97525 106170 31469 106171 31477 126586 106174 40639 106175 126204
87357 118634 88893 89405 89917 126380 56143 126396 95549 118638 97085 128246 98109
127964 126636 126644 126652 60751 118646 126700 126708 126716 62799 118650 126828 63311
122685 106234 40699 106235 40701 126587 106238 40703 106239 127980 127988 127996 128252
87358 87870 88894 89406 97530 88783 93502 128444 95550 122106 126202 97598 98110
128684 121555 128700 111934 56571 128748 128756 128764 95567 89339 121662 122174 97531
96975 126270 97231 128318 128830 97999 130366 130878 130300 129708 129716 129724 46399
56573 129772 129780 129788 87293 89341 56127 95485 97533 129964 129972 129980 130004
130012 130028 130036 130044 87359 87871 88895 89407 97534 121551 93503 130492 95551
122110 126206 97599 98111 130732 129747 130748 111935 56575 130796 130804 130812 120127
89343 121663 122175 97535 129743 126271 131004 128319 128831 129855 130895 131023 87722
96938 87786 95914 129706 87738 93610 87802 120490 122538 97194 128682 130730 87726
88794 87742 112042 96941 87790 87485 87806 120234 96942 89530 122282 129710 89786
126378 97258 128426 128938 129962 130474 130986 87725 128362 109930 109403 120238 87741
128363 125658 88942 95917 97965 120250 87789 120251 87797 114011 120254 87805 120255
96954 120554 95918 97966 95930 121706 89534 120494 122542 122602 128686 130734 125802
89790 126314 96957 126698 126826 89854 96958 128746 128874 97967 129770 129898 130026
130410 128939 130794 130922 131050 125786 88795 109934 125787 95662 112046 128430 126810
95663 129882 130906 112058 121562 129754 93614 112346 88943 125806 97197 121563 129755
128442 95989 120558 126318 130414 97013 95677 113499 97269 95678 112062 128446 98037
95679 129883 130907 122554 97198 97710 130746 95934 97005 89531 89533 89535 89787
122298 122286 129774 126382 122618 122555 129966 130478 130990 89789 89791 126394 95933
97981 126682 126714 126842 97262 89853 89855 129786 128730 128762 122558 122603 97711
130750 97021 129978 130010 130042 130490 128875 129790 130778 130810 122559 131002 131034
131066 112347 121710 128731 95981 129902 97261 126398 98029 130779 95997 128750 126830
122299 130926 97725 97213 126838 97277 122302 130494 98045 128766 98237 122303 131035
126702 130798 122622 129982 130030 126718 131054 130014 122619 122623 131006 130046 130814 131070
\end{lstlisting}


\end{document}